\documentclass[english,twoside]{article}
\usepackage{cite}
\usepackage{amsmath,amssymb}
\usepackage{latexsym}
\usepackage[english]{babel}
\usepackage[latin1]{inputenc}
\usepackage{amsthm}
\usepackage[dvips]{graphicx}
\usepackage{pstricks}
\usepackage{pst-text}
\usepackage{pstcol}

\newcommand{\N}{\mathbb{N}}
\newcommand{\Z}{\mathbb{Z}}
\newcommand{\Q}{\mathbb{Q}}
\newcommand{\R}{\mathbb{R}}

\newcommand{\A}{\mathbb{A}}
\newcommand{\Ao}{\widetilde{\mathbb{A}}}

\newcommand{\tore}{\mathbb{T}}
\newcommand{\T}[1]{\tore^{#1}}

\newcommand{\abs}[1]{\left|#1\right|}
\newcommand{\norma}[1]{\left\|#1\right\|}
\newcommand{\fonc}[3]{#1:#2\to#3}
\newcommand{\sui}[1]{\left( #1_n \right)_{n\in\N}}
\newcommand{\suii}[2]{\left( #1_#2 \right)_{#2\in\N}}

\DeclareMathOperator{\lrho}{\rho_{loc}}
\DeclareMathOperator{\arho}{\rho_{ann}}

\DeclareMathOperator{\adhe}{Cl}

\DeclareMathOperator{\inte}{Int}
\DeclareMathOperator{\diam}{diam}
\DeclareMathOperator{\dist}{dist}
\DeclareMathOperator{\sing}{Sing}
\DeclareMathOperator{\supp}{Supp}

\DeclareMathOperator{\homeo}{Homeo}

\DeclareMathOperator{\conv}{Conv}
\DeclareMathOperator{\fron}{Fr}

\newcommand{\homeourd}{\homeo_0\left(\R^2;0\right)}
\newcommand{\homeoa}{\homeo_0\left(\A\right)}

\theoremstyle{theorem}
\newtheorem{lemm}{Lemma}[section]
\newtheorem{coroll}[lemm]{Corollary}
\newtheorem{propo}[lemm]{Proposition}
\newtheorem{theor}[lemm]{Theorem}
\newtheorem*{claim}{Claim}

\newtheorem{defini}[lemm]{Definition}
\newtheorem*{theoa}{Theorem A}
\newtheorem*{theoaast}{Theorem A*}
\newtheorem*{theob}{Theorem B}
\newtheorem*{theobast}{Theorem B*}
\newtheorem*{theoc}{Theorem C}
\newtheorem*{theocast}{Theorem C*}
\newtheorem*{theod}{Theorem D}

\theoremstyle{remark}
\newtheorem*{exam}{Example}

\newtheorem{remar}{Remark}

\title{\textbf{The Local Rotation Set is an Interval}}
\author{\textsc{Jonathan Conejeros}}
\date{}
\begin{document}

\maketitle

\begin{abstract}
Let $\homeourd$ be the set of all homeomorphisms of the plane isotopic to the identity and which fix $0$. Recently in \cite{leroux} Frédéric Le Roux gave the definition of the \textit{local rotation set of an isotopy $I$} in $\homeourd$ from the identity to a homeomorphism $f$ and he asked if this set is always an interval. In this article we give a positive answers to this question and to the analogous question in the case of the open annulus.
\end{abstract}

\section{Introduction}
The concept of \textit{rotation number} was introduced by H. Poincaré to study the dynamics of orientation-preserving homeomorphisms of the circle $\T{1}$ (in the context of torus flows, see \cite{poin}). More precisely, for every orientation-preserving homeomorphism $h$ of the circle $\T{1}$, H. Poincaré defined a number $\rho(h)$, measuring the ``asymptotic speed of rotation of the orbits of $h$ around $\T{1}$''. He proved that this number provides information about the dynamics of $h$. The construction of H. Poincaré can be generalized for homeomorphisms of surfaces. In the case of the two-dimensional torus $\T{2}=\R^2/\Z^2$, this notion of rotation set was introduced by M. Misiurewicz and K. Ziemian (see \cite{mizi}). For a torus homeomorphism $f$ which is isotopic to the identity, the rotation set associated to some lift $\widetilde{f}$ of $f$, denoted $\rho(\widetilde{f})$, is a subset of $\R^2$. In this, and in many other articles it is studied the relation between the rotation set and the dynamics of $f$. For example if $(p_1/q,p_2/q)$ is a rational point in the interior of $\rho(\widetilde{f})$, then there exists a $q$-periodic point for $f$ (see \cite{fra88}). In \cite{lm91} is proved that if $\rho(\widetilde{f})$ has interior non-empty, then $f$ has positive entropy. More recently, using transverse foliations to the dynamics, one can give more precise descriptions of the dynamics of some homeomorphisms of the torus (see for example \cite{dav13} and \cite{ltal}).

\subsection{Main results}

In this article, we are interested in the local case and in the case of the open annulus $\A=\T{1}\times \R$. In the local case, firstly, V. Na\u\i shul$'$ proved that for diffeomorphisms $f$ which fix 0 and whose differential at $0$ is a rotation, the angle of the rotation is a topological invariant (see \cite{nai}). In \cite{gp95} J.-M. Gambaudo and E. Pécou gave a simple proof of Na\u\i shul$'$ result. Next P. Le Calvez (see \cite{lec03}) and then P. Le Calvez and J. C. Yoccoz (see \cite{ly97}) defined the
\textit{local rotation number} of some homeomorphisms. These works provide settings where the local rotation set is a single number. More recently, F. Le Roux defined \textit{the local rotation set} for a general homeomorphism isotopic to the identity of the plane which fixes $0$ (see \cite{leroux}). An easy but important example is the family of the \textit{fibered rotations}. For a continuous function  $\fonc{\alpha}{(0,+\infty)}{\R}$ we define the fibered rotation as:
$$ h_{\alpha}: (r,\theta)\mapsto (r,\theta+ \alpha(r)).$$  In this case, the local rotation set of $h_\alpha$ coincides with the limit points of the function $\alpha$ at $0$. Hence for every compact interval in $[-\infty,+\infty]$ there exists a fibered rotation whose local rotation set is this interval.

A natural question is: Is the local rotation set always an interval?  The following theorem give a positive answer to the above question. We denote by $\homeourd$ the set of all homeomorphisms of the plane which are isotopic to the identity and which fix 0.

\begin{theoa}
  Let $I$ be an isotopy in $\homeourd$ starting from the identity. Then the local rotation set of $I$, $\lrho(I)$, is an interval. More precisely: every rational number $\frac{p}{q}$ belonging to the interior of the convex hull of $\lrho(I)$ is included in the rotation set of a compact invariant set $K$ arbitrarily close to $0$.
\end{theoa}

 Let us turn to the case of the open annulus $\A=\T{1}\times \R$. The notion of rotation set was introduced by J. Franks (see \cite{fra96}, see also \cite{leroux}). On the other hand, P. Le Calvez introduced the \textit{rotation set of recurrent points} (see \cite{lec01}). Moreover these sets coincide if the latter is non-empty and the homeomorphism satisfies the intersection property (see \cite{wang}). Our second result is the following. We denote by $\homeoa$ the set of all homeomorphisms of the open annulus which are isotopic to the identity.

\begin{theob}
  Let $I$ be an isotopy in $\homeoa$ starting from the identity. Then the rotation set in the open annulus, $\arho(I)$, is an interval.
\end{theob}

\subsection{Analogous results for some compact surfaces}

Let us recall briefly analogous results of Theorems A and B for some compact surfaces. In the case of  the closed annulus $\overline{\A}=\T{1}\times [0,1]$, one can easily give a positive answer to above question. Let $f$ be a homeomorphism of the closed annulus which is isotopic to the identity and let $I$ be an isotopy from the identity to $f$. There is a lift $\fonc{\widetilde{f}}{\R\times [0,1]}{\R\times [0,1]}$ of $f$ associated to $I$. The \textit{rotation set of $I$} is a compact interval of $\R$, which measures the asymptotic speeds of rotation of the orbits of $f$ around the annulus. More precisely let $n\geq 1$ be an integer, we define
$$ \rho_{n}:=  \left\{ \frac{p_1(\widetilde{f}^{n}(\widetilde{z}))-p_1(\widetilde{z})}{n} : \widetilde{z}\in \R\times [0,1]\right\},$$
where $\fonc{p_1}{\R\times [0,1]}{\R}$ is the projection on the first coordinate. The rotation set of $I$ is
$$\rho(I):=\bigcap_{m\geq 1}\adhe(\bigcup_{n\geq m} \rho_n).$$
Since $\overline{\A}$ is a connected set, the set $\rho_{n}$ is a connected subset of $\R$. Using that $\rho_{kn}$ is contained in $\rho_n$ for every integer $k\geq 1$, we deduce that $\cup_{n\geq m} \rho_{n}$ is an interval for all integer $m\geq 1$, and thus $\rho(I)$ being a decreasing sequence of intervals, it is an interval.\\

In the case of the two-dimensional torus $\T{2}=\R^2/\Z^2$, M. Misiurewicz and K. Ziemmian proved in \cite{mizi} that the rotation set is a compact convex subset of $\R^2$. However, it is not known which compact and convex subsets of the plane can be realized as rotation sets of homeomorphisms of $\T{2}$. It was conjectured by J. Franks and M. Misiurewicz in \cite{fm90} that a line segment $L$ could not be realized as a rotation set of a torus homeomorphisms in the following cases: (i) $L$ has irrational slope and a rational point in its interior, (ii) $L$ has rational slope but no rational points and (iii) $L$ has irrational slope and no rational points. Recently P. Le Calvez and F. A. Tal proved case (i) (see \cite{ltal}), and A. \'Avila has announced a counter-example for case (iii). \\

At first sight, the local case and the open annulus case may seem very close to the closed annulus case. However, if we try the straightforward adaptation of the notion of rotation set of the closed annulus, the invariance by topological conjugacy fails. A solution for this is ``to select good orbits'', but the resulting definition is a little more complex (see the definition of $\rho_{K}(I)$ below) and the easy argument for the case of the closed annulus fails to prove the connectedness of the local rotation set.

\subsection{Strategy of the Proofs}
Now, we will give the precise definitions and the strategy of the proof of Theorem A (these can be adapted to prove Theorem B).  Let $f$ be in $\homeourd$ i.e. a homeomorphism of the plane $\R^2$ which is isotopic to the identity and which fixes $0$. Let $I=(f_t)_{t\in [0,1]}$ be an isotopy from the identity to $f$ in $\homeourd$, that is a continuous path $t\mapsto f_t$ from $[0,1]$ to $\homeourd$ (endowed with the compact-open topology), which connects the identity to $f_1=f$. Let $\fonc{\pi}{\R\times (0,+\infty)}{\R^2\setminus \{0\}}$ be the universal covering of $\R^{2}\setminus \{0\}$ and let $\widetilde{I}=(\widetilde{f}_t)_{t\in [0,1]}$ be the lift of the isotopy $I$ to the universal covering of $\R^{2}\setminus \{0\}$ such that $\widetilde{f}_0=Id$, one defines $\widetilde{f}:=\widetilde{f}_1$. Given $z\in \R^2\setminus \{0\}$, and an integer $n\geq 1$ we consider $\rho_n(z)$ the \textit{average change of angular coordinate along the trajectory of $z$ for the isotopy $I$}, i.e.
$$ \rho_n(z):=\frac{1}{n}(p_1(\widetilde{f}^n(\widetilde{z}))-p_1(\widetilde{z})),$$
where $\fonc{p_1}{\R\times (0,+\infty)}{\R}$ is the projection on the first coordinate and $\widetilde{z}$ is a point in $\pi^{-1}(z)$.
Next for a compact set $K$ in $\R^2\setminus \{0\}$  we define
  $$ \rho_{K}(I):=\bigcap_{m\geq 1} \adhe\left( \bigcup_{n\geq m} \{ \rho_n(z) : z\in K,\, f^{n}(z)\in K   \}   \right).   $$
where the closure is taken in $\overline{\R}:=\R\cup \{+\infty\}\cup \{-\infty\}$.
We will give the definition of the \textit{local rotation set of $I$} due to F. Le Roux (see \cite{leroux}).
Given two neighborhoods $V$ and $W$ of $0$ with $W\subset V$, we define
$$ \rho_{V,W}(I):=\bigcap_{m\geq 1} \adhe\left( \bigcup_{n\geq m} \{ \rho_n(z) :  z\notin W, f^n(z)\notin W, \text{ and } z,\cdots,f^{n}(z)\in V \}   \right). $$
The \textit{local rotation set of the isotopy $I$} is defined as
$$ \lrho(I):=\bigcap_{V} \adhe\left( \bigcup_{W\subset V}  \rho_{V,W}(I)   \right),$$
where $W\subset V$ are neighborhoods of $0$.
Note that $\rho_{V,W}(I)$ is analogous to the above definition of $\rho_K(I)$ with $K=V\setminus W$.
Using usual properties of the rotation sets, we can reduce the proof of Theorem A to the following theorem.

\begin{theoaast}
  Let $I$ be an isotopy in $\homeourd$ from the identity to a homeomorphism $f$. Suppose that the local rotation set $\lrho(I)$ contains both positive and negative real numbers. Then $0$ belongs to $\lrho(I)$. More precisely, for every $V$ neighborhood of $0$, there exists an $f$-invariant and compact set $K$ in $\R^2\setminus \{0\}$ contained in $V$, such that $0$ belongs to $\rho_K(I)$.
\end{theoaast}

This theorem is an improvement of a result due to F. Le Roux. Under the hypotheses of Theorem A$^*$ and assuming furthermore that $f$ satisfies the \textit{local intersection property}, i.e. every Jordan curve $\gamma$ non-contractible in $\R^2\setminus \{0\}$ meets its image under $f$, i.e. $f(\gamma)\cap \gamma\neq \emptyset$, he proved that $V$ contains a \textit{contractible fixed point $z$ under $I$}, i.e. a fixed point of $f$ such that its trajectory under the isotopy $I$, $z\mapsto f_t(z)$ is a loop contractible in $\R^2\setminus \{0\}$ (see \cite{leroux} and for the case of the open annulus \cite{lec05}). Thus we see that, in this case, we take $K=\{z\}$.\\

On the other hand, let us examine the opposite case when $f$ does not have any contractible fixed point close enough to $0$. We can apply a generalization of Brouwer's translation theorem due to P. Le Calvez (see \cite{lec05}), and we obtain an oriented foliation $\mathcal{F}$ defined in $\R^2\setminus\{0\}$ which is positively transverse to the isotopy $I$. Let $\gamma$ be a closed leaf of $\mathcal{F}$. Then $\gamma$ is a Jordan curve, it is \textit{essential} i.e. non-contractible in $\R^2\setminus \{0\}$, and it is  \textit{$f$-free}, i.e. $f(\gamma)\cap\gamma=\emptyset$. Using the fact that $\lrho(I)$ contains both positive and negative real numbers, one obtains the following situation.
In every neighborhood of $0$, there exist three closed leaves $\gamma_0$, $\gamma_1$ and $\gamma_2$ of $\mathcal{F}$ such that:

\begin{itemize}
  \item[(i)] the curve $\gamma_1$ separates $\gamma_0$ et $\gamma_2$.
  \item[(ii)] The set of the points $x$ such that all iterates of $x$ under $f$ remain between $\gamma_0$ and $\gamma_1$ is non-empty and its rotation set is contained in $(0,+\infty)$.
\item[(iii)] The set of the points $x$ such that all iterates of $x$ under $f$ remain between $\gamma_1$ and $\gamma_2$ is non-empty and its rotation set is contained in $(-\infty,0)$.
\end{itemize}

Thus Theorem A$^*$ will be a consequence of the following theorem. This is the main and new result of this article.
For $E$ subset of $\R^2\setminus \{0\}$ we write $\Theta(E)$ for the \textit{maximal invariant set of $E$}, that is, the set of the points $x$ of $E$ such that all iterates of $x$ under $f$ remain in $E$. We remark that $\R^{2}\setminus \{0\}$ is homeomorphic to the open annulus $\A$, and so all above definitions are valid in $\A$.

\begin{theoc}
Let $I$ be an isotopy in $\homeoa$ from the identity to a homeomorphism $f$. Assume that there exist $\gamma_0$, $\gamma_1$ and $\gamma_2$ three essential, pairwise disjoint and $f$-free Jordan curves in the annulus $\A$ such that $\gamma_{1}$ separates $\gamma_0$ of $\gamma_2$. For $i\in\{0,1\}$ let $\Theta(A_i)$ be the maximal invariant set of the annulus $A_i$ delimited by $\gamma_i$ and $\gamma_{i+1}$. Suppose that
\begin{itemize}
  \item[(i)] the sets $\Theta(A_0)$ and $\Theta(A_1)$ are non-empty; and
  \item[(ii)] the set $\rho_{\Theta(A_0)}(I)$ is contained in either $(0,+\infty)$ or $(-\infty,0)$ and $\rho_{\Theta(A_1)}(I)$ is contained in the other one.
\end{itemize}
Let $\Theta(A)$ be the maximal invariant set of $A=A_0\cup A_1$. Then $$0\in \rho_{\Theta(A)}(I).$$
\end{theoc}

\begin{figure}[!h]
  \centering
    \includegraphics{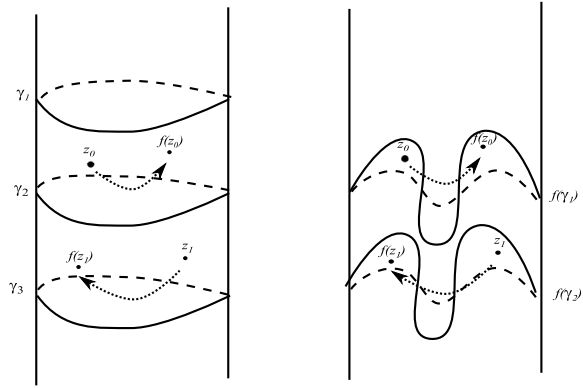}
  \caption{Theorem C}
  \label{fig:theoremeC}
\end{figure}

Another consequence of Theorem C is the following.

\begin{theod}
Let $I$ be an isotopy in $\homeoa$ from the identity to a homeomorphism $f$. Assume that there exist $\gamma^+$ and $\gamma^-$ two essential, disjoint $f$-free Jordan curves in the annulus $\A$. Let $\Theta(A)$ be the maximal invariant set of the annulus $A$ delimited by $\gamma^+$ and $\gamma^-$ and suppose that $\Theta(A)$ is non-empty. Then $\rho_{\Theta(A)}(I)$ is an interval.
\end{theod}

Let us give a plan of this article. In Section 2, we will recall some definitions and results that we will use in the proof of Theorems A and B, in particular P. Le Calvez's generalization of Brouwer's translation theorem. In Section 3, we will prove Theorem A assuming Theorem C. In Section 4, we will prove Theorem B assuming Theorem C. In Section 5, we will prove Theorem D. Finally, in Section 6, we will prove Theorem C.

\subsection*{Acknowledgements}
I thank F. Le Roux and S. Crovisier for proposing me this problem. I thank them for key comments and ideas that led to main constructs and proofs, and mainly for their constant support and encouragement.

\section{Preliminary results}

In this section, we will recall some definitions and results that we will use in the rest of the article, in particular the generalization of Brouwer's Translation Theorem due to P. Le Calvez.

\subsection{Foliations}

Let $M$ be an oriented surface. By an \textit{oriented foliation with singularities} $\mathcal{F}$ on $M$ we mean a closed set $\sing (\mathcal{F})$, called \textit{the set of singularities of $\mathcal{F}$}, together with an oriented topological foliation $\mathcal{F}'$ on $M\setminus \sing (\mathcal{F})$, i.e. $\mathcal{F}'$ is a partition of $M\setminus \sing (\mathcal{F})$ into connected oriented $1$-manifold (circles or lines) called \textit{leaf of $\mathcal{F}$}, such that for every $z$ in $M\setminus \sing (\mathcal{F})$ there exists a neighborhood $U$ of $z$, called \textit{trivializing neighborhood} and an oriented
preserving homeomorphism $\psi$, called \textit{trivialization chart at $z$} mapping $U$ to an open set $U'$ of $\R^2$ such that $\psi$ maps the foliation
induced by $\mathcal{F}$ in $U$ to the oriented foliation by vertical lines oriented upwards. By a theorem of Whitney (see \cite{whi33} and \cite{whi41}), all foliations with singularities $\mathcal{F}$ can be embedded in a flow,
i.e. $\mathcal{F}$ is the set of (oriented) orbits of some \textit{topological flow} $\fonc{\phi}{M\times \R }{M}$ (where the set of singularities of
$\mathcal{F}$ coincides with the set of fixed points of $\phi$). Therefore, we can define the $\alpha$-limit and $\omega$-limit sets of each leaf of $\mathcal{F}$
as follows: If $l$ is a leaf of $\mathcal{F}$ and $z$ is a point of $l$, then
$$  \omega(l):=\bigcap_{t\in [0,+\infty)} \adhe\{ \phi(z,t') : t'\geq t \} \quad \text{and} \quad   \alpha(l):=\bigcap_{t\in (-\infty,0]}
 \adhe\{ \phi(z,t') : t'\leq t \}  . $$

Let $\mathcal{F}$ be an oriented foliation with singularities on $M$. We say that an arc $\gamma$ in $M\setminus \sing (\mathcal{F})$ is \textit{positively transverse}
to $\mathcal{F}$ if for every $t_0\in [0,1]$, and $\psi$ trivialization chart at $\gamma(t_0)$ the application $t\mapsto p_1(\psi (\gamma(t)))$ is
increasing in a neighborhood of $t_0$, where $p_1$ denotes the projection onto the first coordinate.

\subsection{Isotopies} An \textsf{isotopy} $I=(f_t)_{t\in [0,1]}$ on $M$ is a family $(f_t)_{t\in [0,1]}$ of homeomorphisms of $M$ such that the map
$(z,t)\mapsto f_t(z)$ is continuous on $M\times [0,1]$. This implies that the map $(z,t)\mapsto f^{-1}_t(z)$ is continuons, and then we can define the inverse isotopy $I^{-1}=(f^{-1}_t)_{t\in [0,1]}$. Two isotopies $I=(f_t)_{t\in [0,1]}$ and $I'=(f'_t)_{t\in [0,1]}$ are said to be homotopic if there exists a family $(f_{t,s})_{(t,s)\in [0,1]^2}$ of homeomorphisms of $M$ satisfying
\begin{itemize}
  \item $f_{t,0}=f_t$ for every $t\in [0,1]$;
  \item $f_{t,1}=f'_t$ for every $t\in [0,1]$; and
  \item the map $(z,t,s)\mapsto f_{t,s}(z)$ is continuous on $M\times [0,1]^2$.
\end{itemize}
Let $I=(f_t)_{t\in [0,1]}$ and $I'=(f'_t)_{t\in [0,1]}$ be two isotopies on $M$. Then $I*I'=(g_t)_{t\in [0,1]}$ denotes the isotopy defined as
$$ g_t= \begin{cases}
  f_{2t}, & \text{ if } 0\leq t \leq \frac{1}{2}\\
  f'_{2t-1}\circ f_1, & \text{ if } \frac{1}{2}\leq t \leq 1.
\end{cases} $$

In particular if $q$ is an integer, we define $I^q$ as the isotopy $I*\cdots *I$ ($q$ times) if $q \geq 1$ and as $I^{-1}*\cdots *I^{-1}$ ($-q$ times) if $q\leq -1$. Let $M$ be an oriented surface, and let $I=(f_t)_{t\in [0,1]}$ be an isotopy from $f_0=Id_M$ to a homeomorphism $f_1=f$. Let $\fonc{\pi}{\widetilde{M}}{M}$ be the universal covering of $M$, and let
$\widetilde{I}=(\widetilde{f}_t)_{t\in [0,1]}$ be the lift of the isotopy $I$ such that $\widetilde{f}_0=Id_{\widetilde{M}}$. The homeomorphism
$\widetilde{f}=\widetilde{f}_1$ is called the \textit{lift of $f$ associated to $I$}. Let $z\in M$, we call \textit{trajectory of $z$ along $I$} the arc
$t\mapsto f_t(z)$. A fixed point of $f=f_1$ is said to be \textit{contractible for $I$} if its trajectory under $I$ is a loop homotopic to a constant loop in $M$. This definition depends only on the lift $\widetilde{f}$ associated to $I$. In fact, it is easy to see that the set of
contractible fixed points of $I$ coincides with the image of the set of fixed points of $\widetilde{f}$ under $\pi$.

\subsection{The existence of a transverse foliation}
Let $I=(f_t)_{t\in [0,1]}$ be an isotopy from the identity to a homeomorphism $f_1=f$. We say that an oriented topological foliation $\mathcal{F}$ on $M$ is \textit{transverse to $I$}, if for every $z\in M$, there exists an arc which is homotopic to the trajectory of $z$ along $I$ and is positively transverse to $\mathcal{F}$. We know that ``$\mathcal{F}$ is transverse to $I$'' persists on small perturbations of $\mathcal{F}$ (see Lemma 3.1.3 from \cite{leroux}). More precisely, let $\mathcal{F}$ be an oriented topological foliation on $M$. Let $\fonc{\psi}{[0,1]^2}{M}$ be a ``compact chart'', that is an injective map, which maps each vertical line oriented upwards in $[0,1]^2$ to a leaf of $\mathcal{F}$. For every real number $\epsilon>0$, we write $\mathcal{O}(\psi,\epsilon)$ the set of foliations $\mathcal{F}'$ which admit a compact chart $\psi'$ satisfying
$$ \sup_{t\in [0,1]^2} d_M(\psi(t),\psi'(t))<\epsilon, $$
where $d_M$ is a metric on $M$ which induces the topology of $M$. A family of compact charts $(\psi_\alpha)_\alpha$ is said to be locally finite in $M$ if every point $x\in M$ belongs only a finite numbers of $\psi_\alpha([0,1]^2)$ (more details can be found in \cite{leroux}).

\begin{propo}[F. Le Roux, Lemma 3.1.3 \cite{leroux}]\label{proptranversmallpertur}
  Let $I$ be an isotopy on $M$. Let $\mathcal{F}$ be an oriented topological foliation on $M$ transverse to $I$. Suppose that there exist a family $(\psi_\alpha)_\alpha$ of compact charts of $\mathcal{F}$ which is locally finite in $M$, and a positive real sequence $(\epsilon_\alpha)_\alpha$. Then every oriented foliation $\mathcal{F}'$ on $M$ belonging to the Whitney' open set
$$ \mathcal{O}((\psi_\alpha),(\epsilon_\alpha))= \bigcap_\alpha \mathcal{O}(\psi_\alpha,\epsilon_\alpha) $$
  is positively transverse to $I$ too.
\end{propo}

The following theorem due to P. Le Calvez is a generalization of Brouwer's Translation Theorem.

\begin{theor}[P. Le Calvez, Theorem 1.3 \cite{lec05}]\label{theoexisfoliation}
Let $I$ be an isotopy on $M$ from the identity to a homeomorphism $f$.
  If $I$ does not have any contractible fixed point, then there exists an oriented foliation $\mathcal{F}$ which is transverse to $I$.
\end{theor}

\begin{remar}
By a remark due to P. Le Calvez (see \cite{lec05}), the oriented foliation $\mathcal{F}$ is transverse to the isotopy $I$ if and only if in the universal covering, each leaf $\widetilde{l}$ of the lift $\widetilde{\mathcal{F}}$ of the foliation $\mathcal{F}$ is an oriented \textit{Brouwer's line} for $\widetilde{f}$ the lift of $f$ associated to $I$. More precisely, by Schoenflïes' Theorem, there exists an orientation preserving homeomorphism $h$ of $\R^2$ such that $h\circ \widetilde{l}(t)=(0,t)$, for every $t\in \R$. Then we write $R(\widetilde{l})=h^{-1}((0,+\infty)\times \R)$ and $L(\widetilde{l})=h^{-1}((-\infty,0)\times \R)$. We say that $\widetilde{l}$  is an oriented Brouwer's line for $\widetilde{f}$ if $\widetilde{f}^{-1}(\widetilde{l})\subset L(\widetilde{l})$ and $\widetilde{f}(\widetilde{l})\subset R(\widetilde{l})$.
\end{remar}

Since the set of contractible fixed points is usually non-empty, one needs some additional modifications before using the previous theorem. In the local case, we can use a result of F. Le Roux.

\begin{theor}[F. Le Roux, Appendix A \cite{leroux}]\label{theoextention}
  Let $I$ be an isotopy in $\homeourd$ from the identity to a homeomorphism $f$. Suppose that $I$ does not have any contractible fixed point on a neighborhood $U$ of $0$. Then there exists a homeomorphism $\bar{f}$ of the plan which coincides with $f$ in a neighborhood $U'\subset U$ of $0$, and which does not have any contractible fixed point in $\R^{2}\setminus \{0\}$.
\end{theor}

\subsection{Dynamics of the transverse foliation with two singularities}\label{sectdynamicfoliation}

In this subsection, we consider the open annulus $\A=\T{1}\times \R$. We denote by $N$ (resp. $S$) the upper (resp. lower) end of $\A$ and by $\hat{\A}$ the end compactification of $\A$ which is a topological space homeomorphic to the two-dimensional sphere.
Suppose that the isotopy
$I'=(f'_t)_{t\in [0,1]}$ from the identity to $f$ has no contractible fixed points. Thus, we can apply Theorem \ref{theoexisfoliation}, and we obtain
an oriented topological foliation $\mathcal{F}'$ on $\A$ which is transverse to $I'$. Then one can extend $I'$ to an isotopy $I=(f_t)_{t\in [0,1]}$ on $\hat{\A}$ that fixes the ends of $\A$, that is $f_t(N)=N$ and $f_t(S)=S$ for every $t\in [0,1]$.
Similarly, the foliation $\mathcal{F}'$ can be extended to an oriented topological foliation with singularities $\mathcal{F}$ on $\hat{\A}$ where the
singularities are the ends of $\A$. In this case P. Le Calvez proved that $\mathcal{F}$ is also \textit{locally transverse to $I$} at each end of $\A$ (see \cite{lec08}), that is, for every neighborhood $\hat{U}$ of an end $\hat{z}_0$ of $\A$ there exists a neighborhood $\hat{V}$ contained in $\hat{U}$ such that for every
$z\in \hat{V}\setminus \{\hat{z}_0\}$ there is an arc contained in $\hat{U}$ which is positively transverse to $\mathcal{F}$ and homotopic in $\hat{U}$ to the trajectory of $z$ along $I$.

\begin{propo}[P. Le Calvez, Proposition 3.4 \cite{lec08}]\label{theolocallytrans}
The oriented foliation $\mathcal{F}$ is locally transverse to $I$ at each end of $\A$.
\end{propo}

On the other hand, we can apply the theorem of Poincaré-Bendixson to $\mathcal{F}$ seen as a flow (see \cite{pm82}).

\begin{theor}[Poincaré-Bendixson' Theorem]\label{theopoinbendi}
  Let $\mathcal{F}$ be an oriented foliation on $\hat{\A}$ whose singularities are the ends of $\A$. For every leaf $l$ of $\mathcal{F}$ one of the following possibilities holds:
  \begin{itemize}
    \item[(i)] the $\omega$-limit set of $l$ is a closed leaf of $\mathcal{F}$,
    \item[(ii)] the $\omega$-limit set of $l$ is reduced to an end of $\A$, or
    \item[(iii)] the $\omega$-limit set of $l$ is constituted of ends of $\A$ and leaves $l'$ joining two ends (eventually the same) of $\A$.
   \end{itemize}
   One has the same possibilities for the $\alpha$-limit set of $l$. Moreover for every no closed leaf either
  \begin{itemize}
    \item[(iv)] the sets $\omega$(l) and $\alpha(l)$ are reduced to the same end of $\A$, or
    \item[(v)] the sets $\omega$(l) and $\alpha(l)$ are disjoints.
     \end{itemize}
\end{theor}

We recall the following fact (see for example section 3.2 of \cite{hae62}).

\begin{itemize}
  \item[(1)] The union of the closed leaves of $\mathcal{F}$ is closed.
\end{itemize}

By Poincaré-Bendixson Theorem and the transversality of $\mathcal{F}$ and $I$, we have the following properties:

\begin{itemize}
  \item[(2)] every closed leaf $\gamma$ separates $N$ and $S$, i.e. $N$ and $S$ belong to distinct connected components of $\hat{\A}\setminus \gamma$, and
  \item[(3)] every closed leaf $\gamma$ is $f$-free, i.e. $f(\gamma)\cap \gamma =\emptyset$.
\end{itemize}

 The following result is due to F. Le Roux (see \cite{ler00} and \cite{ler04}).
\begin{itemize}
  \item[(4)] The foliation $\mathcal{F}$ has recurrence, that is there is a leaf $l$ of $\mathcal{F}$ such that $\alpha(l)\cup \omega(l)$ is not contained in $\{N,S\}$, or
  \item[(5)] there is a leaf $l$ of $\mathcal{F}$ which joins $N$ and $S$, i.e. either $\alpha(l)=N$ and $\omega(l)=S$ or $\alpha(l)=S$ and $\omega(l)=N$.
\end{itemize}

We deduce the following properties in a neighborhood of an end $\hat{z}_0$ of $\A$. We say that $\hat{z}_0$ is accumulated by closed leaves of $\mathcal{F}$, if every neighborhood of $\hat{z}_0$ contains a closed leaf of $\mathcal{F}$.\\

\textbf{Non-accumulated case:} Suppose that $\hat{z}_0$ is not accumulated by closed leaves of $\mathcal{F}$. Then there exists a leaf $l$ of $\mathcal{F}$ such that either $\alpha(l)=\hat{z}_0$ or $\omega(l)=\hat{z}_0$.
\begin{proof}
  Firstly, suppose that $\mathcal{F}$ has closed leaves. Let $z\in\A$ closed enough to $\hat{z}_0$ and let $l_z$ be the leaf of $\mathcal{F}$ containing $z$. Items (iv) and (v) from Theorem \ref{theopoinbendi}, imply that $\alpha(l_z)$ or $\omega(l_z)$ contains $\hat{z}_0$. Thus $l_z$ or the leaf $l$ given by item (iii) from Theorem \ref{theopoinbendi} has its $\omega$-limit set or $\alpha$-limit set reduced  to $\hat{z}_0$.\\
Next, suppose that $\mathcal{F}$ has recurrence, but not closed leaves.  Let $l$ be a leaf of $\mathcal{F}$ such that $\alpha(l)\cup \omega(l)$ is not contained in $\{N,S\}$. Hence by Theorem \ref{theopoinbendi} either $\alpha(l)$ or $\omega(l)$ contains $\hat{z}_0$. One concludes, as in the first part. \\
Finally, if $\mathcal{F}$ does not have recurrence, then by Property (4) there exists a leaf of $\mathcal{F}$ whose $\alpha$-limit set or $\omega$-limit set is reduced to $\hat{z}_0$. This completes the proof.
\end{proof}

\textbf{Accumulated case:} Suppose that the end $\hat{z}_0$ is accumulated by closed leaves of $\mathcal{F}$. Using a chart centered at $\hat{z}_0$ we can suppose that $\mathcal{F}$ is a foliation on a neighborhood of $0\in\R^2$. Let us fix a closed leaf $\gamma$ of $\mathcal{F}$ and put $\hat{U}$ the closure of the connected component of complement of $\gamma$ containing $0$. Let $\mathcal{C}$ be the union of all closed leaves of $\mathcal{F}$ contained in $\hat{U}$. This is a closed set of $\hat{U}\setminus \{0\}$. The closed leaves are totally ordered by the relation of inclusion of the connected component of their complements containing $0$. By Schoenflies' Theorem, we can construct a homeomorphism of the plane, which maps each closed leaf of $\mathcal{F}$ to a Euclidean circle centered at $0$. Let $A$ be a connected component of the complement (in $\hat{U}$) of $\mathcal{C}\cup \{0\}$. The set $A$ is delimited by two closed leaves $\partial^+A$ and $\partial^-A $, and for every leaf $l$ of $\mathcal{F}$ included in $A$ we have either $\alpha(l)=\partial^+A$ and $\omega(l)=\partial^-A$  or $\alpha(l)=\partial^-A$ and $\omega(l)=\partial^+A$. Moreover all these leaves have the same $\alpha$-limit and $\omega$-limit sets. Hence, the foliation on the closure of $A$ is homeomorphic to a Reeb component or a spiral on $\T{1}\times [0,1]$.

\section{Local Case}

In this section we will introduce the definition of the local rotation set and we will show Theorem A assuming Theorem C.

\subsection{Definitions}

\subsubsection{The centered plane}
We denote by $\R^2$ the plane endowed with its usual topology and orientation. We denote by $0$ the point $(0,0)$ of the plane $\R^2$. Let $\gamma$ be a non-contractible Jordan curve in $\R^2\setminus \{0\}$. We call \textit{interior of $\gamma$} (resp. \textit{exterior of $\gamma$}) the connected component of $\R^2\setminus \gamma$ which contains (resp. does not contain) $0$.

\subsubsection{The set $\homeourd$}

We say that a homeomorphism $f$ of the plane \textit{fixes} 0, if $f(0)=0$. We denote by
$\homeourd$ the set of all homeomorphisms of the plane isotopic to the identity and which fix $0$. We recall that the fundamental group of $\homeourd$ is
isomorphic to $\Z$ (see \cite{ham74}). More precisely, consider $J=(R_t)_{t\in [0,1]}$ the isotopy defined by:
 $$  R_t(x,y):=(x\cos (2\pi t)+y \sin (2\pi t), - x \sin( 2\pi t)+y\cos (2\pi t)). $$

 If $I=(f_t)_{t\in [0,1]}$ and $I'=(f'_t)_{t\in [0,1]}$ are two isotopies in $\homeourd$ both from the identity to the same homeomorphism  $f_1=f'_1$, then there exists a unique integer $q\in\Z$ such that $I'$ is homotopic (with fixed endpoints) to $J^q * I$.

\subsubsection{The local rotation set}\label{sectiondeflocalrotationset}

In this paragraph, we consider a homeomorphism $f$ in $\homeourd$. Let $I=(f_t)_{t\in [0,1]}$ be an isotopy in $\homeourd$ from the identity to $f$. We will give the definition of the \textit{local rotation set of $I$} due to F. Le Roux (see \cite{leroux}).
 Let $\fonc{\pi}{\R\times (0,+\infty)}{\R^2\setminus \{0\}}$ be the universal covering of $\R^{2}\setminus \{0\}$ and let $\widetilde{I}=(\widetilde{f}_t)_{t\in [0,1]}$ be the lift of the isotopy $I$ to the universal covering of $\R^{2}\setminus \{0\}$ such that $\widetilde{f}_0=Id$, one defines $\widetilde{f}:=\widetilde{f}_1$. Given $z\in \R^2\setminus \{0\}$, and an integer $n\geq 1$ we consider $\rho_n(z)$ the \textit{average change of angular coordinate along the trajectory of $z$ for the isotopy $I$}, i.e.
$$ \rho_n(z):=\frac{1}{n}(p_1(\widetilde{f}^n(\widetilde{z}))-p_1(\widetilde{z})),$$
where $\fonc{p_1}{\R\times (0,+\infty)}{\R}$ is the projection on the first coordinate and $\widetilde{z}$ is a point in $\pi^{-1}(z)$. Given two neighborhoods $V$ and $W$ of $0$ with $W\subset V$, we define  the \textit{local rotation set of $I$ relative to $V$ and $W$} by
$$ \rho_{V,W}(I):=\bigcap_{m\geq 1} \adhe\left( \bigcup_{n\geq m} \{ \rho_n(z) : z\notin W,\, f^n(z)\notin W, \text{ and } z,\cdots, f^n(z)\in V    \}   \right),   $$
where the closure is taken in $\overline{\R}:=\R\cup \{+\infty\}\cup \{-\infty\}$. Next we define the \textit{local rotation set of $I$ relative to $V$}
by
$$ \rho_{V}(I):=\adhe\left( \bigcup_{W} \rho_{V,W}(I)    \right)   $$ where $W$ is a neighborhood of $0$ included in $V$.
Finally we define the \textsf{local rotation set of the isotopy $I$} by
$$ \lrho(I):=\bigcap_{V} \rho_{V}(I) .$$
where $V$ is a neighborhood of $0$
\begin{remar}
  If we try to generalize the notion of rotation set of the closed annulus to the local case, we will define
  $$  \bigcap_{V}  \bigcap_{m\geq 1} \adhe\left( \bigcup_{n\geq m} \{ \rho_n(z) : f^{i}(z)\in V \text{ for all } i\in \{0,\cdots, n\}   \}   \right) .$$
  However, in this setting the invariance by conjugation fails, for example the two contraction mapping $z\mapsto \frac{1}{2}z$ and $z\mapsto \frac{i}{2}z$ are conjugate but they will to have $\{0\}$ and $\{\frac{1}{4}\}$ by local rotation set respectively.
\end{remar}

\begin{remar}
  The local rotation set $\lrho(I)$ only depends on the choice of the homotopy class of $I$, i.e. if $I_1$ and $I_2$ are two isotopies from the identity
  to the same homeomorphism which are homotopic, then $\lrho(I_1)=\lrho(I_2)$.
\end{remar}

The local rotation set verifies the following properties.

\begin{propo}\label{propensrotlocal}
 Let $I$ be an isotopy in $\homeourd$ from the identity to $f$. For every integer $p$ let $J^p=(R_{pt})_{t\in [0,1]}$. Then
  \begin{enumerate}
    \item The local rotation set of $I$ is invariant by oriented topological conjugation.
    \item For every $p,q\in\Z$, we have  $\lrho(J^p * I^q)=q\lrho(I)+p$.
  \end{enumerate}
\end{propo}

\begin{remar}\label{remarkpropOin}
One can define the following stronger and stronger notions:
\begin{itemize}
  \item[(1)] Every neighborhood of $0$ contains a contractible fixed point $z\neq 0$ of $I$.
  \item[(2)] Every neighborhood of $0$ contains a compact and $f$-invariant set $K$ not containing $0$ such that $0$ belongs to $\rho_{K}(I)$.
\item[(3)] Every neighborhood $V$ of $0$ contains a compact (non-invariant) set $K$ such that $\rho_K(I)\subset \rho_V(I)$ and $0$ belongs to $\rho_{K}(I)$.
\end{itemize}
Note that $(1)\Rightarrow (2) \Rightarrow (3)$, and $(3)$ implies that $0$ belongs to $\lrho(I)$.
\end{remar}

\subsection{Main result: Proof of Theorem A}\label{subsprooftheoA}

The purpose of this paragraph is to deduce Theorem A from Theorem C. \\
  Let $I$ be an isotopy in $\homeourd$ from the identity to a homeomorphism $f$. Since the local rotation set $\lrho(I)$ is closed and $\Q$ is dense in $\R$, in order to prove Theorem A it is sufficient to show that every
 irreducible rational number $\frac{p}{q}$ which belongs to the interior of the convex hull of $\lrho(I)$ belongs to the local rotation set. Moreover, by
 considering $J^p * I^q$ instead of $I$, we may assume that $p=0$ and $q=1$. This follows from the formula relating the local rotation sets of $I$ and
 $J^p * I^q$ (Proposition \ref{propensrotlocal}). Thus, we reduced the proof of Theorem A to proving the following theorem.

\begin{theoaast}
  Let $I$ be an isotopy in $\homeourd$ from the identity to a homeomorphism $f$. Suppose that the local rotation set $\lrho(I)$ contains both positive and negative real numbers. Then $0$ belongs to $\lrho(I)$. More precisely, for every $V$ neighborhood of $0$, there exists an $f$-invariant and compact set $K$ in $\R^2\setminus \{0\}$ contained in $V$, such that $0$ belongs to $\rho_{K}(I)$.
\end{theoaast}

\begin{remar}
  Theorem A$^*$ implies that Property $(2)$ of Remark \ref{remarkpropOin} holds.
\end{remar}

\begin{exam}
There is a homeomorphism $f$ of $\R^2$ which satisfies Property $(2)$ but not Property $(1)$ of Remark \ref{remarkpropOin}. To construct $f$, we start with a ``twice'' Reeb's homeomorphism $f'_1$ defined on the compact annulus $D_0=\{z\in \R^2  : 1\leq \norma{z} \leq 3 \}$ in $\R^2$ satisfying:
\begin{itemize}
  \item[(i)] $f'_1$ fixes the circles $c_i$, where $i\in\{1,2,3\}$ and $c_i= \{z\in \R^2  : \norma{z}=i \}\subset D_0$.
  \item[(ii)] $f'_1|_{c_1}$ and $f'_1|_{c_3}$ act as a rotation $R_{\beta}$ with angle $\beta>0$ and  $f'_1|_{c_2}$ acts as a rotation $R_{\alpha}$ with angle $\alpha<0$.
  \item[(iii)] $f'_1$ does not have fixed points in $D_0$.
\end{itemize}
Let $I'=(f'_t)_{t\in [0,1]}$ be the natural isotopy on $D_0$ from the identity of $D_0$ to $f'_1$. We define $\fonc{f_t}{\R^2}{\R^2}$ by
$$ f_t(z)=\begin{cases}
  \frac{f'_t(3^n z)}{3^n}, &
  \text{ if } \frac{1}{3^n} \leq \norma{z} \leq \frac{1}{3^{n-1}}, n\in \Z \\ f_t(0)=0
\end{cases} $$
Let $I=(f_t)_{t\in [0,1]}$ be the isotopy from the identity to $f=f_1$. For every integer $n\in \N$, let us put $D_n=\frac{1}{3^n}D_0$. Then $0$ belongs to $\rho_{D_n}(I)$ but $f$ does not have any  fixed point $z\neq 0$.
\end{exam}

For the proof of Theorem A$^*$ one introduces the following property.

\begin{defini}
  We say that a homeomorphism $f$ in $\homeourd$ satisfies the \textsf{local intersection property} if there exists a neighborhood $V$ of $0$, such that
  every Jordan curve $\gamma$ contained in $V$ and non-contractible in $\R^2\setminus \{0\}$ meets its image under $f$, i.e. $f(\gamma)\cap \gamma \neq
  \emptyset$.
\end{defini}

Frédéric Le Roux proved Theorem A* (see \cite{leroux}) in the case when the homeomorphism $f$ satisfies the local intersection property. He proved something better: in this case the local rotation set ``detects'' some fixed points, that is Property (1) of Remark \ref{remarkpropOin} holds. This follows from the following proposition which will be proved in the next subsection.

\begin{propo}\label{propcases1}
  Let $I$ be an isotopy in $\homeourd$ from the identity to a homeomorphism $f$. Suppose that the local rotation set, $\lrho(I)$, contains both positive and negative real numbers. Then one of the following cases holds:\\

    Case (a) Every neighborhood of $0$ contains a contractible fixed point $z\neq 0$.\\

    Case (b) The homeomorphism $f$ does not satisfy the local intersection property, and every neighborhood of $0$ contains three pairwise disjoint,  $f$-free Jordan curves $\gamma_0$, $\gamma_1$ and $\gamma_2$ which are non-contractible in $\R^2\setminus\{0\}$ with $\gamma_{i+1}$ included in the interior of $\gamma_i$ for $i\in \{0,1\}$ and such that the closed annulus $A_i$ delimited by  $\gamma_i$ and $\gamma_{i+1}$ satisfies:
      \begin{itemize}
      \item[(i)] the maximal $f$-invariant set $\Theta(A_i)$ of $A_i$ is non-empty, and
      \item[(ii)] the set $\rho_{\Theta(A_0)}(I)$ is contained in one of the sets $(0,+\infty)$ or $(-\infty,0)$ and the set $\rho_{\Theta(A_1)}(I)$
      is contained in the other one.
    \end{itemize}
\end{propo}

We recall that this proposition is an adaptation of the following result for homeomorphisms of the closed annulus $\T{1}\times [0,1]$. This result was proved by P. Le Calvez for diffeomorphisms in \cite{lec91}, but by Theorem \ref{theoexisfoliation} it is also valid for homeomorphisms.

\begin{theor}[P. Le Calvez \cite{lec91}, \cite{lec05}]
Let $\fonc{f}{\T{1}\times [0,1]}{\T{1}\times [0,1]}$ be a homeomorphism isotopic to the identity with a lift $\fonc{\widetilde{f}}{\R\times [0,1]}{\R\times [0,1]}$ that has no fixed points and whose rotation set contains both negative and positive real numbers. Then there exists a finite non-empty family $\{\gamma_i\}$ of essential, pairwise disjoint and  $f$-free Jordan curves such that the maximal invariant sets of the closed annulus delimited by two consecutive curves has rotation set contained either in $(0,+\infty)$ or in $(-\infty,0)$.
\end{theor}

Now we show that Proposition \ref{propcases1} and Theorem C imply Theorem A$^*$.
\begin{proof}[End of the proof of Theorem A$^*$ (assuming Theorem C)] Let $I$ be an isotopy in $\homeourd$ from the identity to a homeomorphism $f$. Suppose that the local rotation set $\lrho(I)$ contains both positive and negative real numbers. One considers the different cases of Proposition \ref{propcases1}. In the case (a), we are done. On the other hand, if case (b) of Proposition \ref{propcases1} holds, we can apply Theorem C, which says that $0$ belongs to $\rho_{\Theta(A)}(I)$ where $\Theta(A)$ is the maximal invariant set of $A=A_0\cup A_1$. This completes the proof of Theorem A$^*$.
\end{proof}

\subsection{Proof of Proposition \ref{propcases1}}\label{proofofprop33}
Let $I$, $f$ satisfy the hypotheses of the proposition. We prove Proposition \ref{propcases1} assuming that case (a) does not hold, i.e. there is a neighborhood of $0$ that does not contain any contractible fixed point. By Theorem \ref{theoextention}, changing $f$ in the complement of a Jordan domain containing $0$, we may suppose that $f$ does not have any contractible fixed point in $\R^2\setminus\{0\}$. Thus, we can apply Theorem \ref{theoexisfoliation} and we obtain an oriented foliation $\mathcal{F}$ defined in $\R^2\setminus\{0\}$ which by Proposition \ref{theolocallytrans} is locally transverse to $I$ at $0$. The foliation $\mathcal{F}$ can be extended to an oriented topological foliation with singularities on the end compactification of $\R^2\setminus\{0\}$ where the singularities are the ends of $\R^2\setminus\{0\}$.\\

\begin{lemm}
 In this case, the closed leaves of $\mathcal{F}$ accumulate at $0$.
\end{lemm}
 \begin{proof}
Otherwise, from at the end of § \ref{sectdynamicfoliation} there exists a leaf $l$ of $\mathcal{F}$ whose $\omega$-limit or $\alpha$-limit set is reduced to $\{0\}$. The following lemma implies that either $\lrho(I)\subset [0,+\infty]$ or $\lrho(I)\subset [-\infty,0]$. This contradicts our hypotheses.
\end{proof}

\begin{lemm}\label{lemmclosedleafaccu0}
  Let $I$ be an isotopy in $\homeourd$ and let $\mathcal{F}$ be a foliation positively transverse to $I$. Suppose that $\mathcal{F}$ admits a leaf $l$ such that either $\omega(l)=\{0\}$ or $\alpha(l)=\{0\}$. Then either $\lrho(I)\subset [0,+\infty]$ or $\lrho(I)\subset [-\infty,0]$.
\end{lemm}
\begin{proof}
Without loss of generality, we may suppose that $\alpha(l)=\{0\}$. In addition, conjugating $f$ by $\Phi$
(given by Schoenflies' Theorem), we may suppose that the negative  half-leaf $l^-$ containing one point of $l$ is contained in the positive $x$-axis. Let $U$ be a Euclidean circle centered at $0$ whose boundary meets $l^-$. By Proposition \ref{theolocallytrans}, $\mathcal{F}$ is locally transverse to $I$ at $0$. Let $V$ be a neighborhood of $0$ contained in $U$ given by the local transversality of $\mathcal{F}$ to $I$ at $0$: the trajectory of
each $z\in V$ along $I$ is homotopic, with fixed endpoints, to an arc $\alpha_z$ which is positively transverse to $\mathcal{F}$ and is included in $U$.
In particular, the arcs $\alpha_z$ can cross $l^-$ only downwards. More precisely, let $\widetilde{f}$ be the lift of $f$ associated to $I$, and let $\widetilde{\mathcal{F}}$ be the lift of $\mathcal{F}$. Let $\widetilde{U}$ and $\widetilde{V}$ be lifts of the sets $U\setminus\{0\}$ and $V\setminus\{0\}$ respectively. Let $\widetilde{l}^-$ be the lift of $l^-$ contained in the line $\{0\}\times (0,+\infty)$. Let $z\in V\setminus \{0\}$ and let $n\in\N$ such that $\{z,\cdots ,f^{n-1}(z)\}\subset V$. Let $\widetilde{z}$ be the lift of $z$ such that $0<p_1(\widetilde{z})\leq 1$ and $\alpha_{\widetilde{z}}$ the lift of the arc $\alpha_z$ from $\widetilde{z}$. Since the arc $\alpha_{\widetilde{f}^{n-1}({\tilde{z}})}*\cdots
*\alpha_{\widetilde{z}}$ is positively transverse to $\widetilde{\mathcal{F}}$ and does not meet the boundary of $\widetilde{U}$, we obtain that $p_1(\widetilde{f}^n(\widetilde{z}))>0$ and thus
$$\rho_n(z)=\frac{1}{n}(p_1(\widetilde{f}^n(\widetilde{z}))-p_1(\widetilde{z}))\geq -\frac{1}{n}.$$
  This implies the inclusion $\rho_{V}(I)\subset [0,+\infty]$, and so $\lrho(I)\subset [0,+\infty]$. This completes the proof.
\end{proof}

In the sequel, we suppose that the closed leaves of $\mathcal{F}$ accumulate at $0$ and that all closed leaves of $\mathcal{F}$ are Euclidean circles centered at $0$ (see the description at the end of § \ref{sectdynamicfoliation}). We remark that $\mathcal{F}$ cannot coincide with the oriented foliation in Euclidean circles centered at $0$ on a neighborhood of $0$. Otherwise, $f$ is conjugated to a homothety.

 We will consider the set $\mathcal{A}$ of all closed annuli $A$ whose boundary components are closed leaves of $\mathcal{F}$, and which do not contain any closed leaf of $\mathcal{F}$ in their interiors. By Poincaré-Bendixson Theorem (Theorem \ref{theopoinbendi}) the foliation on $A$ is a Reeb component or a spiral.\\

 For every annulus $A\in \mathcal{A}$, we write $\partial^+A$ and $\partial^-A$ its two boundary components with $\partial^-A$ included in the interior of $\partial^+A$.

\begin{defini}
  We say that the foliation $\mathcal{F}$ goes from $\partial^+A$ to $\partial^-A$ (resp. from $\partial^-A$ to $\partial^+A$) on the annulus $A\in \mathcal{A}$, if all the leaves of $\mathcal{F}$ in the interior of $A$ have $\partial^+(A)$ (resp. $\partial^-A$) as their $\alpha$-limit set and $\partial^-(A)$ (resp. $\partial^+A$) as their $\omega$-limit set.
\end{defini}

We have the following lemma.

\begin{lemm}[P. Le Calvez, \cite{lec91}]\label{lemmapropperturbation}
Suppose that the foliation $\mathcal{F}$ goes from $\partial^-A$ to $\partial^+A$ on the annulus $A\in \mathcal{A}$. Let $\Theta(A)$ the maximal invariant set of $A$. Then the set $\rho_{\Theta(A)}(I)$ (eventually empty) is contained in $(0,+\infty)$.
\end{lemm}
\begin{proof}
Since each boundary component of the annulus $A$ is $f$-free, we deduce that $\Theta(A)$ is included in the interior of $A$. For $z\in \Theta(A)$ by transversality of $\mathcal{F}$ and $I$ there exists an arc $\alpha_z$ joining $z$ and $f(z)$, homotopic (with fixed endpoints) to the trajectory of $z$ along the isotopy $I$, and which is transverse to the foliation $\mathcal{F}$. Moreover from the facts that $z$ and $f(z)$ are both in the interior of $A$, we deduce that $\alpha_z$ is also included in the interior of $A$. We can find an annulus $A'$ whose boundary components are Jordan curves not contractible in $\R^2\setminus \{0\}$ such that the set $\Theta(A)=f(\Theta(A))$ is included in $A'$. Considering a change of coordinates, we can assume that the boundary components of $A'$ are Euclidian circles centered at $0$, and that the foliation induced by $\mathcal{F}$ on $A'$ coincides with the radial foliation with the leaves toward $0$. By transversality, for every $z\in\Theta(A)$ and every integer $n\geq 1$, we have that $\rho_n(z)$ is positive. Since $\Theta(A)$ is a compact set, we deduce that $\rho_{\Theta(A)}(I)$ is included in $(0,+\infty)$. This completes the proof.
\end{proof}

In the sequel, $\mathcal{A}'$ denotes the subset of $\mathcal{A}$ consisting of the annuli $A\in \mathcal{A}$ whose maximal invariant set is non-empty. From the above proposition, we can consider the following definition.

\begin{defini}
  We say that $A\in\mathcal{A}$ is a positive annulus if $\rho_{\Theta(A)}(I)\subset (0,+\infty)$. We define a negative annulus similarly.
\end{defini}

\begin{lemm}\label{lemmapropperturbation1}
  Under the hypotheses of Proposition \ref{propcases1}, we have
  \begin{itemize}
    \item[(i)]  the set $\mathcal{A}'$ is locally finite in $\R^2\setminus \{0\}$, i.e. every compact set in $\R^2\setminus \{0\}$ can meet only a finite number of elements of $\mathcal{A}'$; and
    \item[(ii)] both positive and negative annuli accumulate at $0$.
  \end{itemize}
\end{lemm}
\begin{proof}

\textit{Proof of item (i).} Let $\mathcal{A}''$ be the subset of $\mathcal{A}$ constituted of the annuli $A\in \mathcal{A}$ that meet its image under $f$ i.e. $f(A)\cap A\neq \emptyset$. Remark that $\mathcal{A}'\subset \mathcal{A}''$, and so it suffices to prove that the set $\mathcal{A}''$ is locally finite in $\R^2\setminus \{0\}$. This will be a consequence of the following lemma.

\begin{lemm}
Let $\gamma^+$ and $\gamma^-$ be two closed leaves of $\mathcal{F}$ with $\gamma^-$ contained in the interior of $\gamma^+$. Let $A_{\gamma^+,\gamma^-}$ be the closed annulus delimited by $\gamma^+$ and $\gamma^-$. Then there is a finite number of elements of $\mathcal{A}''$ metting $A_{\gamma^+,\gamma^-}$.
\end{lemm}
\begin{proof}
  We know that the space of closed leaves of $\mathcal{F}$ included in $A_{\gamma^+,\gamma^-}$ is compact and that every closed leaf of $\mathcal{F}$ does not meet its image under $f$ (Properties $(1)$ and $(4)$ of §\ref{sectdynamicfoliation}). Hence, we can consider the real number $\epsilon>0 $ defined as:
  $$ \epsilon:=\min_\gamma \dist(\gamma,f(\gamma)), $$
  where $\gamma$ is a closed leaf of $\mathcal{F}$ contained in $A_{\gamma^+,\gamma^-}$ and for $C,D$ two subsets of $\R^2$, $\dist(C,D):=\inf_{c\in C,d\in D}\norma{c-d}$ and $\norma{\cdot}$ is the Euclidean distance in $\R^2$.
Let us prove that all $A\in \mathcal{A}''$ satisfies $\dist(\partial^+A,\partial^-A)\geq \epsilon$. Indeed, if $f(A)\cap A\neq \emptyset$, then either $f(\partial^+A)\cap A\neq \emptyset$ or $f(\partial^-A)\cap A\neq \emptyset$. Suppose that the first case holds (the proof is similar in the second case). Let $z\in \partial^+A$ such that $f(z)\in A$. Hence, there exist real numbers $t^-\leq 1\leq t^+$ such that $z^-=t^- f(z)\in \partial^-A$ and $z^+=t^+ f(z)\in \partial^+A$. Using the fact that $\partial^-A$ and $\partial^+A$ are Euclidean circles centered at $0$, we deduce that
$$ \epsilon\leq \dist(f(\partial^+A),\partial^+A)\leq \norma{f(z)-z^+} \leq \norma{z^--z^+}=\dist(\partial^-A,\partial^+A).$$
Since the annulus $A_{\gamma^+,\gamma^-}$ can be covered by a finite number of annuli $A$ satisfying $\dist(\partial^+A,\partial^-A)\geq \epsilon$ the proof follows.
\end{proof}
This completes the proof of item (i).\\

\textit{Proof of item (ii).} Suppose by contradiction that there exists a neighborhood of $0$ which does not contain any positive annulus (the other case is proved in a similar way). There are two cases.\\

\textit{Case 1:} There exists a neighborhood of $0$ which does not contain any annulus $A$ on which the foliation $\mathcal{F}$ goes from $\partial^-A$ to $\partial^+A$. We can find a small perturbation $\mathcal{F}'$ of $\mathcal{F}$ which ``breaks'' all closed leaves of $\mathcal{F}$ close to $0$. From Proposition \ref{proptranversmallpertur} the foliation $\mathcal{F}'$ is also transverse to $I$ and all leaves of $\mathcal{F}'$ are spirals whose $\omega$-limit set is $\{0\}$. Hence by Lemma \ref{lemmclosedleafaccu0}, the local rotation set $\lrho(I)$ is contained in $[-\infty,0]$, contradicting our hypotheses.\\

\textit{Case 2:} In the other case, every annulus $A$ close enough to $0$ on which $\mathcal{F}$ goes from $\partial^-A$ to $\partial^+A$ have maximal invariant set empty (we are assuming this). We will construct an intermediary foliation $\mathcal{F}_\infty$ transverse to the isotopy $I$ as in the Case 1. The construction has three steps.\\

\textit{1st step:}  We can replace $\mathcal{F}$ on each annulus $A\in \mathcal{A}\setminus \mathcal{A}''$ (annulus which that does not meet its image under $f$), by the foliation in Euclidean circles centered at $0$. We remark that this foliation is also transverse to $I$.\\

\textit{2nd step:}  Recall that by hypothesis every annulus $A$ (included in a small neighborhood of $0$) on which the foliation $\mathcal{F}$ goes from $\partial^-A$ to $\partial^+A$ has maximal invariant set empty. We will use the following lemma.

\begin{lemm}[F. Le Roux, \cite{leroux}]
  Suppose that the maximal invariant set of the annulus $A\in\mathcal{A}$ is empty. Then there exists an annulus $\hat{A}$ containing $A$ whose boundary is the union of $\partial^{+}A$ and $f^n(\partial^-A)$ for some integer $n$, and a foliation $\hat{\mathcal{F}}$ whose leaves $\hat{\gamma}$ in $\hat{A}$ are closed and such that $f(\hat{\gamma})$ is situated on the right of the oriented leaf $\hat{\gamma}$.
\end{lemm}
By the previous lemma, we obtain a foliation in closed leaves on the annulus $\hat{A}$ delimited by $\partial^+A$ and $f^n(\partial^-A)$, this foliation can be pasted with the foliation $\mathcal{F}$ on the exterior of $\partial^+A$, and with the foliation $(f^n)_*\mathcal{F}$ in the interior of $f^n(\partial^-A)$ by to obtain a new foliation $\mathcal{F}_{A,n}$ which is transverse to the isotopy $I$.\\

\textit{3rd step:} The intermediary foliation $\mathcal{F}_\infty$ is constructed as the limit of a sequence of foliations $\sui{\mathcal{F}}$ which is a stationary sequence on every compact set of $\R^2\setminus \{0\}$. By the first step and the fact that $\mathcal{A}'$ is locally finite, the set of annuli $A$ on which $\mathcal{F}$ goes from $\partial^-A$ to $\partial^+A$ accumulate only at $0$ and is a countable set. Let $\sui{A}$ denote  the ordered family of all such annuli, i.e. for every integer $n\geq 1$, the annulus $A_{n-1}$ is contained in the unbounded connected component of $\R^{2}\setminus A_n$.
By hypothesis the maximal invariant set of $A_0$ is empty, so by the second step we obtain a foliation $\mathcal{F}_0=\mathcal{F}_{A_0,n_0}$ which coincides with the foliation in closed leaves on $\hat{A}_0$ (annulus delimited by $\partial^+A_0$ and $f^{n_0}(\partial^-A_0)$ for some integer $n_0$), with $\mathcal{F}$ on the exterior of $\partial^+A_0$, and with the foliation $(f^{n_0})_*\mathcal{F}$ in the interior of $f^{n_0}(\partial^-A)$. In a similar way, we obtain a foliation $\mathcal{F}_1:= (\mathcal{F}_0)_{f^{n_0}(A_1),n_1}$. Iterating this process, we can construct the sequence $\sui{\mathcal{F}}$. Let $\mathcal{F}_\infty$ be the limit of the sequence $\sui{\mathcal{F}}$. Since, by construction each leaf $l_\infty$ of $\mathcal{F}_\infty$ is an iterate of one leaf of $l$ of either $\mathcal{F}$ or the foliation $\hat{\mathcal{F}}$ provides by the previous lemma, i.e. $l_\infty=f^n(l)$, the remark following Theorem \ref{theoexisfoliation} implies that the limit foliation obtained $\mathcal{F}_\infty$ also is transverse to $I$. This completes the proof of item (ii).
\end{proof}

\begin{proof}[End of the proof of Proposition \ref{propcases1}:]
If the case (a) does not hold, then items (i) and (ii) of Lemma \ref{lemmapropperturbation1} hold. Let $V$ be a neighborhood of $0$. Since both negative and positive annuli accumulate $0$ and the collection of them is locally finite, there are a negative annulus $A^-$ and a positive annulus $A^+$ included in $V$ such that the annulus $A$ (eventually an empty set) ``between'' $A^+$ and $A^-$ has maximal invariant set empty. Without loss of generality suppose that $A^-$ is included in the closure of the interior of $\partial^-A^+$. Thus $\gamma_0=\partial^+A^+$, $\gamma_1=\partial^+A^-$ and $\gamma_2=\partial^-A^-$ satisfy the case (b) of the proposition, because $\rho_{A\cup A^-}(I)=\rho_{A^-}(I)$. This completes the proof.
\end{proof}

\section{Case of the Open Annulus}

In this section we will introduce the definition of the rotation set in the open annulus and we will show Theorem B assuming Theorem C.

\subsection{Definitions}

\subsubsection{The open annulus}

We will denote by $\T{1}$ the unit circle of the plane and by $\A:=\T{1} \times \R$ the open annulus. We endow the annulus $\A$ with its usual topology and
orientation. We denote by $N$ (respectively $S$) the upper (respectively lower) end of $\A$ and by $\hat{\A}$ the end compactification of $\A$ which is a
topological space homeomorphic to the two-dimensional sphere. A Jordan curve $\gamma$ in $\A$ is called \textit{essential} if its complement has two
unbounded connected components. We denote by $U_\gamma^N$ the upper one and by $U_\gamma^S$ the lower one. We denote $ \hat{U}_\gamma^N=U_\gamma^N \cup
\{N\}$ and $ \hat{U}_\gamma^S=U_\gamma^S \cup \{S\}$.

\subsubsection{The set $\homeoa$}
We denote by $\homeoa$ the set of all homeomorphisms of the open annulus $\A$ which are isotopic to the identity. We recall that every homeomorphism $f$ in
$\homeoa$ can be extended to a homeomorphism $\hat{f}$ of the end compactification of $\A$ and this homeomorphism fixes both ends of $\A$.

\subsubsection{The rotation set of the annulus}

In this paragraph, we consider a homeomorphism $f$ in $\homeoa$. Let $I$ be an isotopy in $\homeoa$ from the identity to $f$. We will give the definition
of the \textit{rotation set of $I$} due to J. Franks (see \cite{fra96}). 

If $z\in \A$, we consider $\rho_n(z)$ the \textit{average change of angular coordinate along the trajectory of $z$ for the isotopy $I$}. Next for a compact set $K$ in $\A$, we define
  $$ \rho_{K}(I):=\bigcap_{m\geq 1} \adhe\left( \bigcup_{n\geq m} \{ \rho_n(z) : z\in K, \, f^n(z)\in K   \}   \right).   $$
where the closure is taken in $\overline{\R}:=\R\cup \{+\infty\}\cup \{-\infty\}$.

Finally we define the \textsf{rotation set of the isotopy $I$} by
$$ \arho(I):=\adhe\left(\bigcup_{K} \rho_{K}(I) \right)$$
where $K$ is a compact set in $\A$.
\begin{remar}
  The rotation set $\arho(I)$ only depends on the choice of the homotopy class of $I$.
\end{remar}

As above the rotation set verifies the following properties.

\begin{propo}\label{propensrotannulus}
 Let $I$ be an isotopy in $\homeoa$ from the identity to $f$. For every integer $p$ let $J^p=(R_{pt})_{t\in [0,1]}$. Then
  \begin{enumerate}
    \item The rotation set of $I$ is invariant by oriented topological conjugation.
    \item For every $p,q\in\Z$, we have  $\arho(J^p * I^q)=q\arho(I)+p$.
  \end{enumerate}
\end{propo}

\subsubsection{Other rotation sets in the open annulus}

In this paragraph, we introduce other rotation sets in $\A$ such as the \textit{measured rotation set} of an $f$-invariant compact subset of $\A$ and also \textit{the local rotation set relative to a neighborhood} at each end of $\A$.\\

Given a compact and $f$-invariant set $K$ in $\A$, we will denote by $\mathcal{M}_{f}(K)$ the set of all $f$-invariant Borel probability measures whose support is contained in $K$. This is a compact and
convex set in the weak topology. If $\mu\in \mathcal{M}_{f}(K)$, we define its \textit{rotation number} as:
$$ \rho(\mu):=\int_K \rho_1 \, d\mu  .$$

Then, we define the \textit{measured rotation set of $K$} as:

$$ \rho_{mes}(K,I):=\{\rho(\mu) : \mu \in \mathcal{M}_{f}(K) \}  .$$
This set is always a compact interval. More precisely, we have the following result.

\begin{lemm}\label{lemaensK}
Let $K$ be a compact and $f$-invariant set in $\A$. Then
\begin{itemize}
  \item[(1)] The set $\rho_{K}(I)$ is a compact set in $\R$. If $K$ is a connected set, then $\rho_{K}(I)$ is a compact interval.
  \item[(2)] The set $\rho_{mes}(K,I)$ coincides with the convex hull of $\rho_{K}(I)$.
\end{itemize}
\end{lemm}

On the other hand, let $\hat{z}_0$ be an end of $\A$ and let $\hat{V}$ be a neighborhood of $\hat{z}_0$. We consider an oriented preserving homeomorphism $\psi$ mapping $\hat{V}$ to a neighborhood of $0\in\R^2$ such that $\psi(\hat{z}_0)=0$. We write $$ \rho_{\hat{V}}(I):= -\rho_{\psi(\hat{V})}(\psi I \psi^{-1}), $$ where $\rho_{\psi(\hat{V})}(\psi I \psi^{-1})$ is the local rotation set relative to $\psi(\hat{V})$ defined in § \ref{sectiondeflocalrotationset}.

\begin{remar}
  By Proposition \ref{propensrotlocal}, this definition does not depend on the choice of the homeomorphism $\psi$.
\end{remar}

\subsection{Main result: Proof of Theorem B}

The purpose of this subsection is to prove Theorem B assuming Theorem $C$.

As we have seen in § \ref{subsprooftheoA} in order to proof Theorem B, it suffices to prove the following theorem.

\begin{theobast}
  Let $I$ be an isotopy in $\homeoa$ from the identity to a homeomorphism $f$. Suppose that the rotation set $\arho(I)$ contains both positive and negative real numbers. Then there exists a compact (a priori non-invariant) set $K$ in $\A$, such that $0$ belongs to $\rho_{K}(I)$.
\end{theobast}

The following example satisfies Property (3) but not Property (2) from the analogous Remark \ref{remarkpropOin} for the open annulus. This proves that Theorem $B^*$ is sharp, i.e. the compact set $K$ such that $0$ belongs to $\rho_{K}(I)$ can be non-invariant.\\

\begin{exam}
There are a homeomorphism $f$ of $\A$ and an isotopy $I$ from the identity to $f$ satisfying the following conditions:
\begin{itemize}
  \item[(i)] There is no $f$-invariant compact set in $\A$, and
  \item[(ii)] $\arho(I)=[-1,1]$.
\end{itemize}
\end{exam}
Let us fix two real numbers $0<R^-<R^+$. To construct $f$, we start with a homeomorphism $f'_1$ defined on $\R^2\cup \{\infty\}$, the one-point compactification of $\R^2$, satisfying:
\begin{itemize}
  \item[(i)] $f'_1$ acts as a rotation $R_{+1}$ (resp. $R_{-1}$) on the disk $\{z\in\R^{2}: \norma{z}\leq R^-\}$ (resp. $\{z\in\R^{2}: \norma{z}\geq R^+\}$).
  \item[(ii)] $f'_1$ is a Reeb's homeomorphism on $\{z\in\R^{2}: R^-\leq \norma{z}\leq R^+\}$
\end{itemize}
Let us define the following equivalence relation on $\R^2\cup \{\infty\}$.
$$ z=(x,y) \sim \begin{cases}
   0 & \text{ if } \norma{z}\leq R^- \text{ and } y\leq 0;\\
   (0,y) &\text{ if } \norma{z}\leq R^- \text{ and } y\geq 0.\\
\end{cases}$$

$$ z=(x,y) \sim \begin{cases}
   \infty & \text{ if } \norma{z}\geq R^+ \text{ and } y\leq 0;\\
   (0,y) &\text{ if } \norma{z}\leq R^+ \text{ and } y\geq 0.\\
\end{cases}$$

Let $\hat{\A}'=\R^2\cup\{\infty\}/\sim$ be the quotient space of $\R^2\cup \{\infty\}$. By $z\in \R^{2}\cup \{\infty\}$, we write $[z]$ the class of equivalence of $z$.  Now, it is easy to check that $\hat{\A}'\setminus\{[0],[\infty]\}$ is homeomorphic to the open annulus $\A$ et that $f'_1$ induces a homeomorphism $f$ on $\A'$. Moreover $f$ is isotopic to the identity of $\A'$, there exists an isotopic $I$ on the annulus $\A'$ from the identity to $f$ such that $\arho(I)=[-1,1]$.

\subsubsection{Classification}

The proof of Theorem B* consider several cases described by the following proposition. 

\begin{propo}\label{propperturbation2}
  Let $I$ be an isotopy in $\homeoa$ from the identity to a homeomorphism $f$. Suppose that the rotation set $\arho(I)$ contains both positive and negative real numbers. Then one
  of the following cases holds:\\

    Case (a) There exists an invariant and compact set $K$ in $\A$ such that $0$ belongs to $\rho_K(I)$. \\

    Case (b)
      The homeomorphism $f$ satisfies the local intersection property at least at one of the ends of $\A$,
      there exists an essential $f$-free Jordan curve $\gamma$ in $\A$, such that the local rotation set of $I$ at $N$ relative to $\hat{U}_\gamma^N$,  $\rho_{\hat{U}_\gamma^N}(I)$, is contained in one of the sets
      $[0,+\infty]$ or $[-\infty,0]$ and the local rotation set at $S$ relative to $\hat{U}_\gamma^S$, $\rho_{\hat{U}_\gamma^S}(I)$, is contained in the other one.\\

   Case (c) The homeomorphism $f$ satisfies the local intersection property at both ends of $\A$, there exist two $f$-free essential Jordan curves $\gamma_N$ and $\gamma_S$ with $\gamma_N\subset U_{\gamma_S}^N$ such that the closed annulus $A$ delimited by  $\gamma_N$ and $\gamma_S$ satisfies:
      \begin{itemize}
      \item[(i)] the maximal $f$-invariant set of $A$ is not empty, and
      \item[(ii)] the sets $\rho_{\hat{U}_{\gamma_N}^N}(I)$ and $\rho_{\hat{U}_{\gamma_S}^S}(I)$ are contained in one of the sets $[0,+\infty]$ or $[-\infty,0]$ and the set $\rho_A(I)$ is contained in the interior of the other one.
    \end{itemize}

\end{propo}
We prove this proposition assuming that case (a) does not hold. Then, we may assume that $I$ has no contractible fixed points. We apply Theorem \ref{theoexisfoliation}, and we obtain an oriented foliation on the annulus $\A$ which is transverse to $I$. We can see $\mathcal{F}$ as a foliation with singularities on $\hat{\A}$, the end compactification of $\A$, whose singularities are the ends of $\A$.

\begin{lemm}\label{Fhasnoclosedleaves}
  In this situation, there exists an oriented foliation $\mathcal{F}'$ arbitrarily close to $\mathcal{F}$ (in Whitney's topology) which has a closed leaf.
\end{lemm}
\begin{proof}
We claim that $\mathcal{F}$ would have recurrence (see Property (4) of § \ref{sectdynamicfoliation}). Otherwise, by Properties $(4)$ and $(5)$ of § \ref{sectdynamicfoliation}, there exists a leaf of $\mathcal{F}$ joining the two ends of $\A$, this implies that either $\arho(I)\subset [-\infty,0]$ or $\arho(I)\subset [0,+\infty]$. This contradicts our hypotheses. Hence, changing the orientation of $\mathcal{F}$ by its inverse if necessary, we may suppose that there is a leaf $l$ of $\mathcal{F}$ whose $\omega$-limit set contains a point $z\in \A$. Let $l_z$ be the leaf of $\mathcal{F}$ containing $z$. Note that the $\omega$-limit and $\alpha$-limit sets of $l_z$ are reduced to $S$ or $N$. Without loss of generality, we may assume that this end is $S$. Let $U$ be a trivializing neighborhood of $z$. Let $\sui{z}$ be a sequence of points in $l\cap U$ such that for every integer $n\geq 1$, the
point $z_{n+1}$ belongs to $l^+_{z_n}$, the positive half-leaf containing $z_n$, and converges to $z$. For every integer $n\geq 1$, let $l_n$ be the connected component of $l\cap U$ containing $z_n$. Then the sequence $\sui{l}$ converges, in Hausdorff's topology, to the connected component of the leaf $l_z\cap U$ containing $z$. Let $a_n$ and $b_n$ be the first and last point in $l_n$ respectively. Now, we perturb $\mathcal{F}$ enough as Figure \ref{fig:modificationfeuilletage11} This completes the proof of lemma.

\begin{center}
\begin{figure}[h!]
  \centering
    \includegraphics{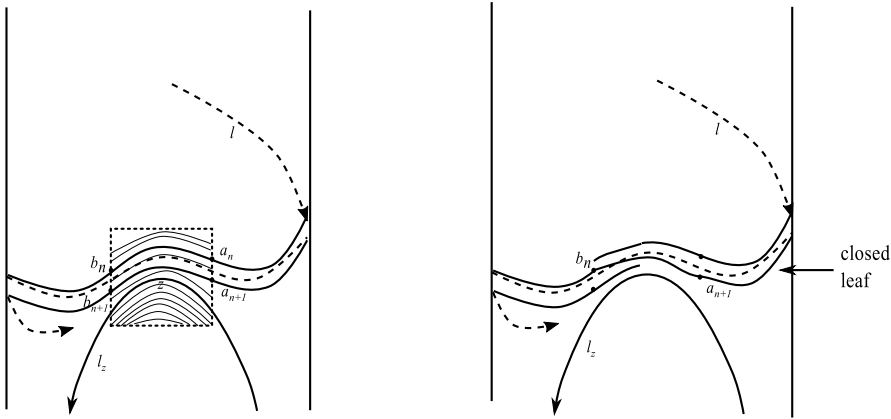}
  \caption{Left: foliation $\mathcal{F}$ Right: foliation $\mathcal{F}'$}
  \label{fig:modificationfeuilletage11}
\end{figure}
\end{center}
\end{proof}

From Proposition \ref{proptranversmallpertur}, the foliation $\mathcal{F}'$ is also transverse to $I$. Let us assume from now on that $\mathcal{F}$ has closed leaves. Moreover, $\mathcal{F}$ cannot have only closed leaves. Otherwise $f$ is conjugated to a homothety.

\begin{lemm}\label{lemma44}
 Under the hypotheses of Proposition \ref{propperturbation2}, assume furthermore that Case (a) does not hold. Then the closed leaves of $\mathcal{F}$ do not accumulate at least at one of the two ends of $\A$.
\end{lemm}
Since the closed leaves of $\mathcal{F}$ are $f$-free essential curves (Properties (2) and (3) from § \ref{sectdynamicfoliation}),
this lemma will follow from Theorem C. As in § \ref{proofofprop33}, we consider the set $\mathcal{A}$ of all closed annuli $A$ whose boundary components are closed leaves of $\mathcal{F}$, and which do not contain any closed leaf of $\mathcal{F}$ in their interior, and $\mathcal{A}'$ the subset of $\mathcal{A}$ consisting of the annuli $A$ whose maximal invariant set is not empty. We say that $A\in \mathcal{A}'$ is a positive (resp. negative) annulus if $\rho_{\Theta(A)}(I)\subset (0,+\infty)$ (resp. $\rho_{\Theta(A)}(I)\subset (-\infty,0)$).

\begin{proof}[Proof of Lemma \ref{lemma44}]
By contradiction, we suppose that the closed leaves of $\mathcal{F}$ accumulate both ends of $\A$. We claim that there are two closed leaves $\gamma^+$ and $\gamma^-$ of $\mathcal{F}$ close enough to the end $N$ and $S$ respectively, such that the closed annulus $A_{\gamma^+,\gamma^-}$ delimited by $\gamma^+$ and $\gamma^-$ contains a positive and a negative annulus. By contradiction, suppose that all annuli of $\mathcal{A}'$ are positive (or negative). Then in a similar way as in Lemma \ref{lemmapropperturbation1} (applied at both ends of $\A$) we can find an oriented foliation $\mathcal{F}_\infty$ which is transverse to $I$, and contains a leaf joining the two ends of $\A$. This contradicts the fact that the rotation set $\arho(I)$ contains both positive and negative real numbers. In a similar way as in end of the proof of Proposition \ref{propcases1} we conclude (assuming Theorem C) that $0$ belongs to $\rho_{\Theta(A_{\gamma^+,\gamma^-})}(I)$, where $\Theta(A_{\gamma^+,\gamma^-})$ is the maximal invariant set of the closed annulus delimited by $\gamma^+$ and $\gamma^-$. This contradicts the fact that case $(a)$ of Proposition \ref{propperturbation2} does not hold. This completes the proof.
\end{proof}

Let us assume from now on, without loss of generality, that this end is $S$. Let $\gamma_S$ be ``the last''
closed leaf of $\mathcal{F}$, that is $\gamma_S$ is a closed leaf of $\mathcal{F}$ such that $U_{\gamma_S}^S$ does not contain any closed leaf of
$\mathcal{F}$. We have the following result.

\begin{lemm}\label{lemmaperturbation2}
  There exists an oriented foliation $\mathcal{F}'$ arbitrarily close to $\mathcal{F}$ (in Whitney's topology) which coincides with $\mathcal{F}$ in
  $\adhe(\hat{U}_{\gamma_S}^N)$ and admits a leaf either whose $\omega$-limit set is reduced to $S$ and $\alpha$-limit set is $\gamma_S$ or  whose $\alpha$-limit set
  is reduced to $S$ and $\omega$-limit set is $\gamma_S$.
\end{lemm}
\begin{proof}
  Let $z$ be an element in $\gamma_S$ and let $\Sigma$ be an arc positively transverse containing $z$ in its interior. Since $\gamma_S$ is a closed leaf of $\mathcal{F}$ and there is no closed leaf in $U_{\gamma_S}^S$, there is a
  leaf $l$ of $\mathcal{F}$ which meets $\Sigma \cap U_{\gamma_S}^S$ at less two times. From Theorem \ref{theopoinbendi}, we deduce that either $\alpha(l)$
  or $\omega(l)$ is $\gamma_S$. Without loss of generality, suppose that we are in the first case. Since there are not closed leaves of $\mathcal{F}$ in
  $U_{\gamma_S}^S$, we have two possibilities (items (ii) or (iii)  of Theorem \ref{theopoinbendi}): the set $\omega(l)$ is reduced to $S$, or there is a leaf $l_S$ of $\mathcal{F}$ from $S$ to $S$ contained in $\omega(l)$. In the first case, the foliation $\mathcal{F}$ satisfies the conclusion of the lemma. In the second case, let $z'$ be an element in $l_S$ and let $U'$ be a trivializing neighborhood of $z'$. Let
   $\sui{z'}$ be a sequence of points in $l\cap U'$ such that for every integer $n\geq 1$, the
point $z'_{n+1}$ belongs to $l^+_{z'_n}$ and converges to $z'$. For every integer $n\geq 1$, let $l_n$ be the connected component of $l\cap U'$ containing $z'_n$. Then the sequence $\sui{l}$ converges, in Hausdorff's topology, to the connected component $l_{z'}'$ of the leaf $l_{z'}\cap U'$ containing $z'$. Let $a_n$ and $b'$ be the first point in $l_n$ and the last point in $l'_{z'}$ respectively. Now, we perturb $\mathcal{F}$ enough as Figure \ref{fig:modificationfeuilletage2} to obtain the expected foliation. This completes the proof of lemma.

\begin{center}
\begin{figure}[h!]
  \centering
    \includegraphics{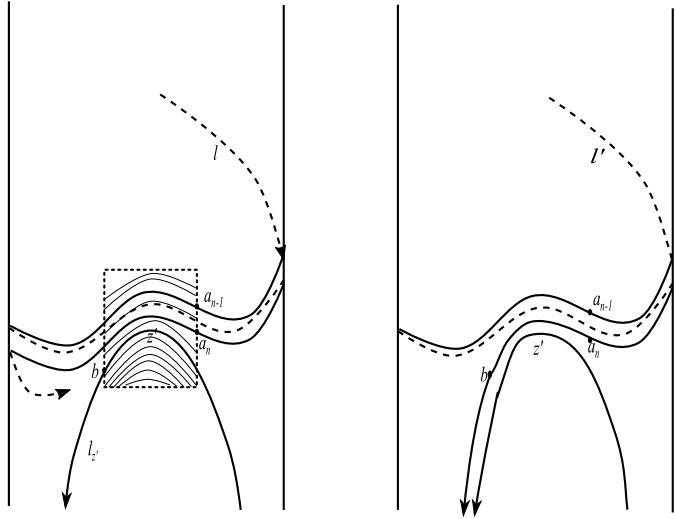}
  \caption{Left: foliation $\mathcal{F}$ Right: foliation $\mathcal{F}'$}
  \label{fig:modificationfeuilletage2}
\end{figure}
\end{center}

\end{proof}

Using analogous reasoning as in Lemma \ref{lemmclosedleafaccu0}, we obtain the following result.

\begin{lemm}\label{lemmaoneside}
  The set $\rho_{\hat{U}_{\gamma_S}^S}(I)$ is contained either in $[0,+\infty]$ or in $[-\infty,0]$.
\end{lemm}

\begin{proof}[\textit{End of the proof of Proposition \ref{propperturbation2}}:]

We have two cases.\\
\textit{Case 1:} The closed leaves of $\mathcal{F}$ do not accumulate at $N$ either. As above, we consider ``the first'' and ``the last'' closed leaves of
$\mathcal{F}$, denoted $\gamma_N$ and $\gamma_S$ respectively, i.e. $\gamma_N\subset U_{\gamma_S}^N$ and there are not closed leaves of $\mathcal{F}$
outside of the closed annulus delimited by $\gamma_N$ and $\gamma_S$. From Lemma \ref{lemmaperturbation2} (applied to both $N$ and $S$), we can assume that the foliation $\mathcal{F}$ admits a leaf $l_N$ whose $\omega$-limit set is reduced to $N$ and $\alpha$-limit set is $\gamma_N$ or  whose $\alpha$-limit set is reduced to $N$ and $\omega$-limit set is $\gamma_N$ and another leaf $l_S$ whose $\omega$-limit set is reduced to $S$ and $\alpha$-limit set is $\gamma_S$ or  whose $\alpha$-limit set is reduced to $S$ and $\omega$-limit set is $\gamma_S$.
To fix idea assume that $l_S$ satisfies $\omega(l_S)=S$ and $\alpha(l_S)=\gamma_S$. We have two subcases.\\

\textit{Subcase 1.1:} The leaf $l_N$ satisfies $\alpha(l_N)=N$ and $\omega(l_N)=\gamma_N$. Let $\Theta(A_{\gamma_N,\gamma_S})$ be the maximal invariant set of $A_{\gamma_N,\gamma_S}$ the closed annulus delimited by $\gamma_N$ and $\gamma_S$. We will prove the following assertions.
\begin{itemize}
  \item[(1)] The set $\mathcal{A}'$ is non-empty,
  \item[(2)] $\mathcal{F}$ goes from $\partial^-A$ to $\partial^+A$ on all annuli $A\in\mathcal{A}'$, and
  \item[(3)] $f$ satisfies the local intersection property at least at one end of $\A$.
\end{itemize}
Let us prove (1). Note that if $\mathcal{A}'$ is empty, then as in the proof of item (ii) of Lemma \ref{lemmapropperturbation1}, there exist an annulus $\hat{A}_{\gamma_N,\gamma_S}$ containing $A_{\gamma_N,\gamma_S}$ whose boundary components are the union of $\gamma_N$ and $f^n(\gamma_S)$, and a foliation  $\mathcal{F}'$ whose leaves in $\hat{A}_{\gamma_N,\gamma_S}$ are closed leaves. This foliation can be pasted with $\mathcal{F}$ on $U^N_{\gamma_N}$, and with $(f^n)_*\mathcal{F}$ on $U_{f^n(\gamma_S)}^S$ by to obtain a new foliation  $\mathcal{F}_{A_{\gamma_N,\gamma_S},n}$  which is transverse to the isotopy $I$. Hence, we can find a small perturbation of $\mathcal{F}_{A_{\gamma_N,\gamma_S},n}$  which is also transverse to $I$ and contains a leaf that joins the two ends of $\A$. This implies that $\arho(I)\subset [-\infty,0]$. This contradicts our hypotheses, completing the proof of (1).\\

Let us prove (2). If $\mathcal{F}$ goes from $\partial^+A$ to $\partial^-A$  on all annuli of $A\in \mathcal{A}'$, as in the proof of Lemma \ref{lemmapropperturbation1}, we can perturb $\mathcal{F}$ by obtain a foliation which admits a leaf joining the two ends of $\A$. This implies that $\arho(I)\subset [-\infty,0]$, contradicting  our hypotheses. On the hand hand, if there exists both an annulus $A^+$ on which $\mathcal{F}$ goes from $\partial^-A^+$ to $\partial^+A^+$ and another $A^-$ on which $\mathcal{F}$ goes from $\partial^+A^-$ to $\partial^-A^-$, Theorem C implies that $0$ belongs to $\rho_{\Theta(A_{\gamma_N,\gamma_S})}(I)$. This is a contradiction, because we are assuming that Case (a) does not hold. This completes the proof of (2).\\

Let us prove (3). Under our suppositions of the leaves $l_N$ and $l_S$ and above assertion (2), we have that $\rho_{\hat{U}^S_{\gamma_S}}(I)\cup \rho_{\hat{U}^N_{\gamma_N}}(I)\subset [-\infty,0]$ and $\rho_{\Theta(A_{\gamma_N,\gamma_S})}(I)\subset (0,+\infty)$. Assertion (3) will is a consequence of the following claim.

\begin{claim}
  Under our hypotheses. Suppose that $f$ does not satisfies the local intersection property at $S$. Then $f$ restricted to $\hat{U}_{\gamma_S}^S$ is a homothety. In particular $\rho_{\hat{U}_{\gamma_N}^S}(I)$ is included in $[0,+\infty]$.
\end{claim}
\begin{proof}[Proof of the claim] Let $\sui{\gamma}$ be a sequence of essential $f$-free curves such that
\begin{itemize}
  \item[(i)] for every integer $n\in \N$, we have $\gamma_{n+1} \subset U_{\gamma_n}^S \subset U_{\gamma_S}^S$, and
 \item[(ii)] $\bigcap_{n\in\N} U_{\gamma_n}^S = \emptyset$.
\end{itemize}
For each integer $n$, we write $\Theta(A_n)$ the maximal invariant set of $A_n$ the closed annulus delimited by $\gamma_S$ and $\gamma_n$. Firstly suppose that there exists a positive integer $n_0$ such that  $\Theta(A_{n_0})$ is not empty. Then $\rho_{\Theta(A_{n_0})}(I)\subset (-\infty,0)$ and we know that $\rho_{\Theta(A_{\gamma_N,\gamma_S})}(I)\subset (0,+\infty)$. Hence Theorem C implies that $0$ belongs to the rotation set relative to the maximal invariant set of the closed annulus delimited by  $\gamma_N$ and $\gamma_{n_0}$. This contradicts again the fact that Case (a) does not hold. Now suppose that for every integer $n$, $\Theta(A_{n})$ is empty. Considering $f$ instead of $f^{-1}$, we can assume that $\gamma_S$ satisfies $f(\gamma_S ) \subset U^S_{\gamma_S }$. Hence for each integer $n$, there is an integer $k(n)$ such that $f^{k(n)}(\gamma_S ) \subset U^S_{\gamma_n }$. Therefore
$$   \bigcap_{n\in\N} f^n(U_{\gamma_S}^S ) \subset \bigcap_{n\in\N} U^S_{\gamma_n }=\emptyset.  $$
This proves that $f$ is a homothety, completing the proof of the claim.
\end{proof}
Let us continue proving assertion (3). Suppose by contradiction that $f$ does not satisfy the local intersection property at none end of $\A$.
Then, the previous claim (applied to both end of $\A$) implies that $\arho(I)$ is contained in $[0,\infty]$. This contradicts our hypotheses, completing the proof of assertion (3).\\

To complete the proof of Proposition \ref{propperturbation2} in the case of Subcase 1.1, note that if $f$ satisfies the local intersection property at the other end of $\A$ then Case (c) holds. Otherwise, by the previous claim Case (b) holds, completing the proof.\\

\textit{Subcase 1.2:} The leaf $l_N$ satisfies $\alpha(l_N)=\gamma_N$ and $\omega(l_N)=N$. Let $\Theta(A_{\gamma_N,\gamma_S})$ be the maximal invariant set of $A_{\gamma_N,\gamma_S}$ the closed annulus delimited by $\gamma_N$ and $\gamma_S$. Then, one and only one of the following conditions holds.
\begin{itemize}
  \item[(1)] $\Theta(A_{\gamma_N,\gamma_S})$ is a non-empty set. Then the set $\mathcal{A}'$ is non-empty, and we can prove as in the proof of the previous subcase that $\mathcal{F}$ goes either from $\partial^-A$ to $\partial^+A$ or from $\partial^+A$ to $\partial^-A$ on all annuli $A\in\mathcal{A}'$.
\item[(2)] $\Theta(A_{\gamma_N,\gamma_S})$ is empty.
\end{itemize}
In both case, as in the proof of Lemma \ref{lemmapropperturbation1}, we can find a perturbation $\mathcal{F}'$ of $\mathcal{F}$ which has an unique closed leaf $\gamma$. From Lemma \ref{lemmaoneside} we know that $\rho_{\hat{U}_\gamma^N}(I)$ is contained in $[0,+\infty]$ and that $\rho_{\hat{U}_\gamma^S}(I)$ is contained in $[-\infty,0]$. Let us prove by contradiction that $f$ satisfies the local intersection property at least at one end of $\A$. Otherwise, since $\arho(I)$ contains both positive and negative real numbers,  we can find two $f$-free essential Jordan curves $\gamma^+$ and $\gamma^-$ close enough to the end $N$ and $S$ respectively, such that the maximal invariant set $\Theta(A_{\gamma^+,\gamma})$ (resp. $\Theta(A_{\gamma,\gamma^-})$) of the annulus delimited by $\gamma^+$ and $\gamma$ (resp. $\gamma$ and $\gamma^-$) is not empty. Moreover $\rho_{\Theta(A_{\gamma_+,\gamma})}(I)$ is contained in $[0,+\infty]$ and $\rho_{\Theta(A_{\gamma,\gamma_-})}(I)$ is contained in $[-\infty,0]$. Therefore Theorem $C$ implies that $0$ belongs to $\rho_{\Theta(A_{\gamma^+,\gamma^-})}(I)$. This contradicts that Case (a) does not hold. This proves that Case (b) holds, completing the proof.\\

This completes the proof of Proposition \ref{propperturbation2} in Case 1.\\

\textit{Case 2:} The closed leaves of $\mathcal{F}$ accumulate at $N$. As above, we consider ``the last'' closed leaf of $\mathcal{F}$ near to $S$, denoted $\gamma_S$. As above from Lemma \ref{lemmaperturbation2} we can assume that there is a leaf $l_S$ that satisfy $\omega(l_S)=S$ and $\alpha(l_S)=\gamma_S$.
 Since $\arho(I)$ contains both negative and positive real numbers the following assertions hold (the proofs are similar as in Case 1).
\begin{itemize}
  \item[(1)] The set $\mathcal{A}'$ is non-empty,
  \item[(2)] $\mathcal{F}$ goes from $\partial^-A$ to $\partial^+A$ on all annuli $A\in\mathcal{A}'$, and
  \item[(3)] $f$ satisfies the local intersection property at $S$.
\end{itemize}
Here again we can find a perturbation $\mathcal{F}'$ of $\mathcal{F}$ which has an unique closed leaf $\gamma$. From Lemma \ref{lemmaoneside} we know that $\rho_{\hat{U}_\gamma^N}(I)$ is contained in $[0,+\infty]$ and that $\rho_{\hat{U}_\gamma^S}(I)$ is contained in $[-\infty,0]$. This proves that Case $(b)$ holds. This completes the proof of Proposition \ref{propperturbation2} in Case 2.
\end{proof}

Proposition \ref{propperturbation2} is now proved.

\subsubsection{An intermediate result}
The purpose of this paragraph is to prove the following proposition and its corollary.

\begin{propo}\label{propcomplement}
  Let $I$, $f$ be as in Theorem B*. Assume furthermore that
  \begin{itemize}
  \item[$(A_1)$] the homeomorphism $f$ satisfies the local intersection property at $S$,
  \item[$(A_2)$] there exists a Jordan domain $\hat{V}$ containing $S$ (in $\hat{\A}$) such that $\hat{f}(\adhe (\hat{V}))\subset \hat{V}$, and
  \item[$(A_3)$] the set $\rho_{\hat{V}}(I)$ is included in $[0,+\infty]$.
\end{itemize}
  Then one of the following assertions holds.
  \begin{itemize}
    \item[(1)] there exists a compact set $K$ in $\A$ such that $0$ belongs to $\rho_K(I)$, or
    \item[(2)] there exists a compact set $K$ in $\A$ contained in $\A\setminus \hat{V}$ such that the set $\rho_K(I)$ intersects $[0,+\infty]$.
  \end{itemize}
\end{propo}

\begin{remar}
  We can replace $(A_2)$ by the hypothesis $\adhe (\hat{V}) \subset \hat{f}(\hat{V})$ by changing the isotopy $I$ by $I^{-1}$. We can replace $(A_3)$ by the
  hypothesis $\rho_{\hat{V}}(I) \subset [-\infty,0]$ by conjugating $I$ by the orientation reversing homeomorphism $(u,r)\mapsto (-u,r)$ of the annulus $\A$.
\end{remar}
Let us put $\Theta^+(\hat{V})$ (resp. $\Theta^-(\hat{V}^c)$) the forward (resp. backward) $\hat{f}$-invariant set of $\hat{V}$ (resp. $\hat{V}^c)$, i.e.:
$$  \Theta^+(\hat{V}):= \bigcap_{n\in \N} \hat{f}^n(\hat{V}) \quad \text{ and } \quad \Theta^-(\hat{V}^c):= \bigcap_{n\in \N} \hat{f}^{-n}(\hat{V}^c).   $$

Using hypothesis $(A_2)$, one proves that the set $\Theta^+(\hat{V})$ (resp. $\Theta^-(\hat{V}^c)$) is a continuum, its complement is a connected set of
$\hat{\A}$ and contains $S$ (resp. $N$). Moreover they satisfy the following properties.
\begin{lemm}\label{lemmatheta+non}
  Under the hypotheses of Proposition \ref{propcomplement}, we have
  \begin{itemize}
    \item[(a)] the set $\Theta^+(\hat{V})$ is not reduced to $S$, and
     \item[(b)] the set $\Theta^-(\hat{V}^c)$ is not reduced to $N$.
  \end{itemize}
\end{lemm}
\begin{proof}
  Let us prove (a). Since $f$ satisfies the local intersection property at $S$, there is an open neighborhood $\hat{U}$ of $S$ which does not contain any $f$-free essential Jordan curve. Hence for every integer $n\geq 0$, the curve $\hat{f}^n(\partial \hat{V})$ is not contained in $\hat{U}$, because it is a free essential Jordan curve by Hypothesis $(A_2)$. Hence and Hypothesis $(A_2)$, $(\hat{f}^n(\adhe(\hat{V})\setminus \hat{U}))_{n\in\N}$ is a decreasing sequence of compact sets in $\A$, and so $$ \emptyset \neq  \bigcap_{n\in\N} \hat{f}^n(\adhe(\hat{V})\setminus \hat{U}) \subset \Theta^+(\hat{V}). $$ This proves item (a).\\

  Now let us prove (b). Suppose by contradiction that $ \Theta^-(\hat{V}^c)=\{N\}$. So, by Hypothesis $(A_2)$, for every compact set $K$ in $\A$, there exists an integer $n_K=n\geq 0$ such that $K$ is included in $\hat{f}^{-n}(\hat{V})$. Hence and by Hypotheses $(A_2)$ and $(A_3)$, we have
  $$  \rho_{K}(I)\subset \rho_{\hat{f}^{-n}(\hat{V})}(I)= \rho_{\hat{V}}(I)\subset [0,+\infty]. $$
This contradicts the fact that $\arho(I)$ contains some negative real number. This proves item (b).
\end{proof}

Recall that $\fonc{\pi}{\widetilde{\A}}{\A}$ is the universal covering of $\A$, and $\fonc{p_1}{\widetilde{\A}}{\R}$ the projection onto
the first coordinate.

\begin{proof}[\textit{Proof of Proposition \ref{propcomplement}:}]
  Let us consider $V:=\hat{V}\cap \A$. Let $\Gamma$, $\widetilde{V}$, $\Theta^+(\widetilde{V})$ and $\Theta^-(\widetilde{V}^c)$ be lifts of $\fron(V)$, $V$,
  $\Theta^+(\hat{V})\setminus\{S\}$ and $\Theta^-(\hat{V}^c)\setminus\{N\}$ respectively.

Let $X$ be the set $\widetilde{\A}\setminus ( \Theta^+(\widetilde{V})\cup \Theta^-(\widetilde{V}^c))$. The set $X$ is open, connected and
$\widetilde{f}$-invariant, i.e. $\widetilde{f}(X)=X$, where $\widetilde{f}$ is the lift of $f$ associated to $I$. Let $\fonc{\alpha}{[0,1]}{\widetilde{\A}}$
be a simple arc satisfying:
\begin{itemize}
  \item[(i)] the endpoints $\alpha(0)$ and $\alpha(1)$ lie in $\Theta^+(\widetilde{V})$ and $\Theta^-(\widetilde{V}^c)$ respectively.
  \item[(ii)] If $\dot{\alpha}$ denotes the interior of $\alpha$, i.e. $\dot{\alpha}$ is the image of $(0,1)$ under $\alpha$, then the set $\dot{\alpha}$ is included in $X$.
  \item[(iii)] The intersection $\alpha \cap \Gamma $ is a singleton.
\end{itemize}
The arc $\alpha$ separates $X$, because $\pi(X)$ is an open annulus and $\alpha$ joins its two ends (see the paragraph that follows the definitions of
$\Theta^+(\hat{V})$ and $\Theta^-(\hat{V}^c$)). Let us denote $R_{X}(\alpha)$ (resp. $L_X(\alpha)$) the connected component of $X\setminus \alpha$ on the
right (resp. left) of $\alpha$ (considering the standard orientation). Since $\dot{\alpha}$ is contained in $X$, and $X$ is $\widetilde{f}$-invariant, one
deduces that for every integer $n\geq 0$, the set $\widetilde{f}^n(\dot{\alpha})$ is included in $X$, and so we have three possibilities:
  \begin{itemize}
    \item $\widetilde{f}^n(\dot{\alpha})\cap \alpha \neq \emptyset$\;;
    \item $\widetilde{f}^n(\dot{\alpha})\subset L_X(\alpha) $; or
    \item $\widetilde{f}^n(\dot{\alpha})\subset R_X(\alpha)$.
  \end{itemize}
To conclude the proof it suffices consider the following three cases.\\

 \textit{Case 1:} Suppose that there exists a sequence of integers $(n_k)_{k\in\N}$ which converges to $+\infty$ such that for all $k\in \N$ we have
 $\widetilde{f}^{n_k}(\dot{\alpha})\cap \alpha \neq \emptyset$. Then $0$ belongs to $\rho_{\pi(\alpha)}(I)$. This proves that Assertion (1) of Proposition \ref{propcomplement} holds.\\

  \textit{Case 2:} Suppose that there exists a sequence of integers $(n_k)_{k\in\N}$ which converges to $+\infty$ such that for all $k\in \N$ we have
  $\widetilde{f}^{n_k}(\dot{\alpha})\subset L_X(\alpha)$. Let $\Gamma^r$ be the connected component of $\Gamma\setminus \alpha$ contained in $R_X(\alpha)$.
  From the facts that $\widetilde{f}^{n_k}(\alpha\cap \Gamma)\in L_X(\alpha)$, $\widetilde{f}^{n_k}(\Gamma^r)$ is a connected set in $X$ which is unbounded
  to the right and contained in $\widetilde{V}$ (by Hypothesis $(A_2)$), we deduce that $$  \widetilde{f}^{n_k}(\Gamma^r)\cap
  \alpha \cap \widetilde{V} \neq \emptyset.$$
  Therefore, there exist a real constant $M>0$ independent of $k$ and a sequence of points $\suii{\widetilde{z}}{k}$, where $\widetilde{z}_k\in \Gamma^r$ for
   all $k\in\N$ such that $$p_1(\widetilde{f}^{n_k}(\widetilde{z}_k))<p_1(\widetilde{z}_k)+M.$$ Hence, if $\rho$ is a limit point of the sequence
   $\left(\rho_{n_k}(\widetilde{z}_k)\right)_{k\in\N}$, where
   $$  \rho_{n_k}(\widetilde{z}_k):=\frac{1}{n_k}(p_1(\widetilde{f}^{n_k}(\widetilde{z}_k))-p_1(\widetilde{z}_k)),$$ then $\rho\leq 0$. This implies
   that $\rho=0$, because $\rho_{\hat{V}}(I)\subset [0,+\infty]$. Thus $0$ belongs to $\rho_K(I)$ where $K=\fron(V)\cup (\pi(\alpha)\cap V)$. This proves that Assertion (1) of Proposition \ref{propcomplement} holds.\\

   \textit{Case 3:} Suppose that there exists a sequence of integers $(n_k)_{k\in\N}$ which converges to $+\infty$ such that for all $k\in \N$ we have
  $\widetilde{f}^{n_k}(\dot{\alpha})\subset R_X(\alpha)$. Let $\Gamma^r$ be the connected component of $\Gamma\setminus \alpha$ contained in $R_X(\alpha)$.
  Since for all  $k\in\N$, the arc $\widetilde{f}^{n_k}(\alpha)$ meets $\Theta^+(\widetilde{V})$ and $\Theta^-(\widetilde{V}^c)$, we deduce that
  $$\widetilde{f}^{n_k}(\dot{\alpha} \cap \widetilde{V}^c)\cap \Gamma^r \neq \emptyset.$$
Therefore, there exist a real constant $M>0$ independent of $k$ and a sequence of points $\suii{\widetilde{z}}{k}$, where
$\widetilde{z_k}\in \alpha \cap \widetilde{V}^c$ such that $$p_1(\widetilde{z}_k)<  p_1(\widetilde{f}^{n_k}(\widetilde{z}_k)) +M.$$ If $\rho$ is a limit
point of the sequence $\left(\rho_{n_k}(\widetilde{z}_k)\right)_{k\in\N}$, where
   $$  \rho_{n_k}(\widetilde{z}_k):=\frac{1}{n_k}(p_1(\widetilde{f}^{n_k}(\widetilde{z}_k))-p_1(\widetilde{z}_k)),$$ then $\rho\geq 0$.
   Hence, if $K$ is the compact set $\fron(V)\cup (\pi(\alpha)\cap V^c)$ of $\A$, then $\rho_K(I)$ contains positive real numbers. This proves that Assertion (2) of Proposition \ref{propcomplement} holds. This completes the proof of Proposition \ref{propcomplement}.
\end{proof}

As a consequence we have the following corollary.

\begin{coroll}\label{coropropcomplement}
  Let $I=(f_t)_{t\in [0,1]}$, $f$ be as in Theorem B*. Assume furthermore that
  \begin{itemize}
  \item[$(A'_1)$] The homeomorphism $f$ satisfies the local intersection property at $S$.
  \item[$(A'_2)$] There exist two $f$-free essential Jordan curves $\gamma_N$ and $\gamma_S$ with $\gamma_N\subset U_{\gamma_S}^N$ such that the closed annulus $A$ delimited by  $\gamma_N$ and $\gamma_S$ has maximal invariant set $\Theta(A)$ non-empty.
  \item[$(A'_3)$] the set $\rho_{\hat{U}_{\gamma_S}^S}(I)$ is included in $[0,+\infty]$ and $\rho_A(I)$ is contained in $[-\infty,0]$.
\end{itemize}
  Then there exists a compact set $K$ in $\A$ such that $0$ belongs to $\rho_K(I)$.
\end{coroll}
\begin{proof}
  Changing $I$ by $I^{-1}$ and conjugating $I$ by the orientation reversing homeomorphism $(u,r)\mapsto (-u,r)$ of the annulus $\A$, we may suppose that $\gamma_N$ is attracting, i.e. $f(\gamma_N)\subset U_{\gamma_N}^S$ and that hypothesis $(A'_3)$ holds. Let us put $A_0=\adhe(U_{\gamma_N}^S)$ and for every integer $n\geq 1$ we define  $A_n=f^{-n}(U_{\gamma_N}^S)$. Note that $\A'=\cup_{n\in\N} A_n$ is an open annulus and $f(\A')=\A'$. Let us put $f'=f\vert_{\A'}$. We note that $f'$ is isotopic to the identity of $\A'$, and since $f'=f\vert_{\A'}$ there is an isotopy $I'$ on the annulus $\A'$ from the identity of $\A'$ to the homeomorphism $f'$ such that for every compact set $K$ contained in $\A'$, we have $\rho_K(I)=\rho_K(I')$. By Proposition \ref{propcomplement}, applied to $I'$ and $f'$ on $\A'$, we know that
  \begin{itemize}
    \item[(1)] there exists a compact set $K'$ in $\A'$ such that $0$ belongs to $\rho_K(I')$, or
    \item[(2)] there exists a compact set $K'$ in $\A'$ contained in $\A'\setminus U_{\gamma_{S}}^S$ such that the set $\rho_{K'}(I')$ intersects $[0,+\infty]$.
  \end{itemize}
If assertion (1) holds we are done. Suppose now that assertion (2) holds. Since $K'$ is a compact set contained in $\A'\setminus U_{\gamma_{S}}^S$, there exists an integer positive $n$ such that $K'$ is included in $A\cup \cdots \cup f'^{-n}(A)$. Hence $\rho_{K'}(I')$ is included in the union of $\rho_{f'^{-i}(A)}(I')$. Since for every $i$, $\rho_{f'^{-i}(A)}(I')= \rho_{A}(I')$  and $\rho_{A}(I)=\rho_{A}(I')$, we deduce that $\rho_{A}(I)$ intersects $[0,+\infty]$. Moreover, by hypothesis, we know that $\rho_A(I)$ is included in $[-\infty,0]$. Therefore, $0$ belongs to $\rho_A(I)$. This completes the proof.
\end{proof}

\subsection{Proof of Theorem B*}

In this subsection we finish the proof of Theorem B*, using Theorem C. Let $I$ an isotopy in $\homeoa$ from the identity to a homeomorphism $f$. Suppose that
$\arho(I)$ contains both positive and negative real numbers. One considers the different cases of Proposition \ref{propperturbation2}. In case (a) we are done.\\
   Suppose now that case (b) of Proposition \ref{propperturbation2} holds. Let $g$ be the orientation reversing homeomorphism of $\A$ defined by
   $(u,r) \mapsto (-u,r)$. Considering $I^{-1}$ instead of $I$, or conjugating $I$ by $g$ or considering $gI^{-1}g^{-1}$ instead of $I$, we can assume that
   hypotheses $(A_1)$-$(A_3)$ of Proposition \ref{propcomplement} hold. Thus, one of the following assertions holds.
   \begin{itemize}
     \item[(1)] There exists a compact set $K$ of $\A$ such that $0$ belongs to $\rho_{K}(I)$, or
     \item[(2)] there exists a compact set $K$ contained in $\adhe(U_{\gamma}^N)$ of $\A$ such that the rotation set $\rho_{K}(I)$ intersects $[0,+\infty]$.
     \end{itemize}
   If assertion (1) holds we are done. Suppose now that assertion (2) holds. Since $K\subset \adhe (U_{\gamma}^N)$ and
   $\hat{f}(\adhe(U_\gamma^S))\subset U_\gamma^S$, we deduce that $\rho_K(I)\subset \rho_{\hat{U}_{\gamma}^N}(I)\subset [-\infty,0]$, this last inclusion holds because case (b) of Proposition \ref{propperturbation2} and hypothesis $(A_3)$ of Proposition \ref{propcomplement} hold. Therefore $0$ belongs
   to $\rho_K(I)$. This completes the proof when the case (b) of Proposition \ref{propperturbation2} holds. \\
 At last suppose that Case (c) of Proposition \ref{propperturbation2} holds. Now, by Corollary \ref{coropropcomplement}, there exists a compact set $K$ of $\A$ such that $0$ belongs to $\rho_{K}(I)$. This completes the proof of Theorem B*.

\section{Proof of Theorem D}\label{sectheoremD}

The purpose of this section is to prove Theorem D assuming Theorem C. Let $I$, $f$, $\gamma^+$ and $\gamma^-$ be as in the hypotheses of the theorem. Without loss of generality, we suppose that $\gamma^-$ is contained in $U_{\gamma^+}^S$. Let $\Theta(A)$ be the maximal invariant set of the annulus $A$ delimited by $\gamma^+$ and $\gamma^-$. As we have seen in § \ref{subsprooftheoA}, we can suppose that $\rho_{\Theta(A)}(I)$ contains both positive and negative real numbers. We want to prove that $0$ belongs to $\rho_{\Theta(A)}(I)$.\\

An $f$-free essential Jordan curve $\gamma$ in $\A$ will be said to be attracting if $f(\gamma)\subset U_\gamma^S$ and repulsing if $f^{-1}(\gamma)\subset U_\gamma^S$. The proof of the next lemma is easy.

\begin{lemm}\label{lemmaattrrepuls}
Let $\gamma$ and $\gamma'$ be two disjoint $f$-free essential curves in $\A$. Suppose that $\gamma$ is attracting and $\gamma'$ is repulsing. Then the maximal invariant set of the closed annulus delimited by $\gamma$ and $\gamma'$ is a connected set.
\end{lemm}

We remark that Lemma \ref{lemaensK} implies Theorem D when $\Theta(A)$ is a connected set. By the previous lemma and considering $I^{-1}$ instead of $I$ (if necessary), we can assume that both $\gamma^+$ and $\gamma^-$ are attracting Jordan curves. Hence the sets
$$  \Theta^+(\hat{U}_{\gamma^-}^S):= \bigcap_{n\in \N} \hat{f}^n(\hat{U}_{\gamma^-}^S) \quad \text{ and } \quad \Theta^-(\hat{U}_{\gamma^+}^N):= \bigcap_{n\in \N} \hat{f}^{-n}(\hat{U}_{\gamma^+}^N)$$ are continuum, their complements are connected set of $\hat{\A}$ and contain $S$ and $N$ respectively. Let us write $\A'$ the set $\hat{\A}\setminus ( \Theta^+(\hat{U}_{\gamma^-}^S)\cup \Theta^-(\hat{U}_{\gamma^+}^N) )$. Then $\A'$ is homeomorphic to an open annulus and $f(\A')=\A'$. Let us write $f'=f\vert_{\A'}$. We note that $f'$ is isotopic to the identity of $\A'$, and since $f'=f\vert_{\A'}$ there is an isotopy $I'$ on the annulus $\A'$ from the identity of $\A'$ to the homeomorphism $f'$ such that for every compact set $K$ contained in $\A'$, we have $\rho_K(I)=\rho_K(I')$. Since $f'=f\vert_{\A'}$ does not satisfy the local intersection property at neither $N'$ nor $S'$, where $N'$ and $S'$ are the ends of $\A'$, by Proposition \ref{propperturbation2} there exists an $f'$-invariant and compact set $K$ in $\A'$ such that $0$ belongs to $\rho_K(I')$. Since $K$ is an invariant set in $\A'$ under $f'$, we deduce that $K$ is included in $\Theta(A)$. Therefore,
$$ 0\in \rho_{K}(I')\subset \rho_{\Theta(A)}(I')=\rho_{\Theta(A)}(I).  $$
This proves Theorem D.

\section{Dynamics in the Closed Annulus: Theorem C}

The purpose of this section is to prove Theorem C. Let $I$, $f$, $\gamma_0$, $\gamma_1$ and $\gamma_2$ be as in the hypotheses of the theorem. Without loss of generality, for $i\in \{0,1\}$ we suppose that $\gamma_{i+1}$ is contained in $U_{\gamma_i}^S$. Let $\Theta(A_i)$ be the maximal invariant set of the closed annulus $A_i$ delimited by $\gamma_i$ and $\gamma_{i+1}$. Let $\Theta(A)$ be the maximal invariant set of the annulus $A=A_0\cup A_1$. We want to prove that $0$ belongs to $\rho_{\Theta(A)}(I)$.

\subsection{Reduction of Theorem ~C}

We recall that Lemma \ref{lemaensK} implies Theorem C when $\Theta(A)$ is a connected set. Thus, from now on we assume that $\Theta(A)$ is a disconnected set. This implies that the $f$-free essential Jordan curves $\gamma_0$, $\gamma_1$ and $\gamma_2$ are either all attracting or all repulsing (see § \ref{sectheoremD})  . Indeed, changing $I$ into $I^{-1}$, we can suppose that $\gamma_0$ is attracting. If $\gamma_2$ is repulsing, then by Lemma \ref{lemmaattrrepuls} $\Theta(A)$ is a connected set. If $\gamma_1$ is repulsing and $\gamma_2$ attracting, then by Lemma \ref{lemmaattrrepuls} both $\Theta(A_0)$ and $\Theta(A_1)$ are connected sets which separate the annulus. In this case, $\Theta(A)$ is the union of $\Theta(A_0)$, of $\Theta(A_1)$ and of the bounded connected component of the complement of $\Theta(A_0)\cup \Theta(A_1)
$ and so $\Theta(A)$ is a connected set. Therefore, changing $I$ into $I^{-1}$ and conjugating it by a change of coordinates (given by Schoenflies' Theorem), we can assume that $\gamma_0$, $\gamma_1$ and $\gamma_2$ are attracting Euclidean circles.\\

We will work in $\widetilde{\A}:=\R\times \R$ the universal covering of $\A=\T{1}\times \R$ where $\T{1}=\R/\Z$. We recall that $\fonc{\pi}{\widetilde{\A}}{\A}$ is the canonical projection of $\widetilde{\A}$ onto $\A$ and that $\fonc{p_1}{\widetilde{\A}}{\R}$ is the projection onto the first coordinate. Let $\widetilde{f}$ be the lift of $f$ associated to $I$, that is, if
 $\widetilde{I}=(\widetilde{f}_t)_{t\in [0,1]}$ is the unique lift of $I$ such that $\widetilde{f}_0$ is the identity of $\Ao$, then
 $\widetilde{f}:=\widetilde{f}_1$. Recall that $\fonc{\widetilde{f}}{\Ao}{\Ao}$ is a homeomorphism which is isotopic to the identity (this is equivalent to require that $\widetilde{f}$ is orientation preserving) and which commutes with the translation $\fonc{T}{\Ao}{\Ao}$ defined by $T(x,y)=(x+1,y)$. For $j\in\{0,1,2\}$ we define $\Gamma_j:=\pi^{-1}(\gamma_j)$ and let $\Theta(\widetilde{A}):=\pi^{-1}(\Theta(A))$. If $\gamma$ is an essential Jordan curve in $\A$, and $\Gamma=\pi^{-1}(\gamma)$, we will write $U_{\Gamma}^+$ (resp. $U_{\Gamma}^-$) the upper (resp. lower) unbounded connected component of the complement of $\Gamma$. From now on, we will write $\rho(\widetilde{f})$ instead of $\arho(I)$.\\

Since $\rho(\widetilde{f}^q)=q\rho(\widetilde{f})$ for every $q\in\Z$, if $\Theta(\widetilde{A})$ is a disconnected and non-empty set, then considering
an iterate of $\widetilde{f}$ instead of $\widetilde{f}$, we can assume the following hypotheses:

\begin{itemize}
  \item[($H_1$)] The sets $\Gamma_0$, $\Gamma_1$ and $\Gamma_2$ are horizontal straight lines such that for $i\in\{0,1\}$, $\Gamma_{i}\subset U^+_{\Gamma_{i+1}}$ and thus are attracting, that is $\widetilde{f}(\Gamma_j)\subset U^-_{\Gamma_j}$ for $j\in\{0,1,2\}$.
  \item[($H_2$)] For every integer $n\geq 1$, we have $ \widetilde{f}^n(\Gamma_0)\cap \Gamma_2\neq \emptyset$.
\end{itemize}
For $i\in \{0,1\}$, we define by $\Theta(\widetilde{A}_i)$ the maximal invariant set of $\widetilde{A}_i$ the topological closed band delimited by $\Gamma_i$ and $\Gamma_{i+1}$.
\begin{itemize}
  \item[$(H_3)$] The sets $\Theta(\widetilde{A}_0)$ and $\Theta(\widetilde{A}_1)$ are non-empty and
  $$ \rho_{\Theta(\widetilde{A}_0)}(\widetilde{f})\subset (0,+\infty) \quad \text{ and } \quad \rho_{\Theta(\widetilde{A}_1)}(\widetilde{f})\subset (-\infty,0). $$
  \end{itemize}
Thus, we have reduced Theorem C to the following theorem.

\begin{theocast}
  Let $\fonc{\tilde{f}}{\Ao}{\Ao}$ be a homeomorphism isotopic to the identity, commuting with the translation $T$, and satisfying hypotheses $(H_1)$ to $(H_3)$. Let $\Theta(\widetilde{A})$ be the invariant maximal set of $\widetilde{A}:=\widetilde{A}_0\cup \widetilde{A}_1$. Then  $$0\in \rho_{\Theta(\widetilde{A})}(\widetilde{f}). $$
\end{theocast}

\paragraph*{Outline of the proof of Theorem C*}

In Subsection \ref{subsecbranches}, using hypotheses $(H_1)$ and $(H_2)$, we begin by constructing unstable sets and
 stable sets of the band $\widetilde{A}_i$ delimited by $\Gamma_i$ and $\Gamma_{i+1}$. Then we consider their connected components,
 called unstable branches and stable branches: these are connected and compact sets, contained in $\widetilde{A}_i$ and meeting $\Gamma_{i+1}$ but not $\Gamma_i$
 and $\Gamma_i$ but not $\Gamma_{i+1}$ respectively.
Let $\Lambda_0^-(x)$ denote the unstable branch of $x\in \Theta(\widetilde{A}_0)$ for the band $\widetilde{A}_0$ and let $\Lambda_1^+(y)$ denote the stable branch of $y\in \Theta(\widetilde{A}_1)$ for the band $\widetilde{A}_1$.\\
Next, in Subsection \ref{subsecconsehypH3}, we will study the sequence of  iterates of  $\Lambda_0^-(x)$. Using hypothesis $(H_3)$ (and in particular that $0$ does not belong to $\rho_{\Theta(\widetilde{A}_0)}(\widetilde{f})$), we will prove the following properties:
\begin{itemize}
  \item[(1)] For every compact set $\widetilde{K}$ included in $\widetilde{A}_0$, $\widetilde{f}^n(\widetilde{K})$ is disjoint from $\widetilde{K}$ for every large enough integer $n$.
  \item[(2)] The sequence $(\widetilde{f}^{-n}(\Lambda_0^-(x))_{n\in\N}$ is uniformly bounded to the right.
\end{itemize}

Finally, in Subsection \ref{subsecendprooftheoC}, we will finish the proof of Theorem $C^*$. From the previous properties Theorem C$^*$ we will be a consequence of the following proposition.
\begin{propo}\label{prop61}
 For every $x\in \Theta(\widetilde{A}_0)$ there exists $y\in \Theta(\widetilde{A}_1)$ such that for every sufficiently large integer $n$, we have
$$ \widetilde{f}^n(\Lambda_0^-(x))\cap \Lambda_1^+(y)\neq \emptyset. $$
\end{propo}
Note that if $\widetilde{z}\in \widetilde{f}^n(\Lambda_0^-(x))\cap \Lambda_1^+(y)$, then $z\in \Theta(\widetilde{A})$. By above Property (2) and its symmetric statements the positive and negative iterates of $\widetilde{z}$ go to the left. This proves that $0$ belongs to $\rho_{\Theta(\widetilde{A})}(\widetilde{f})$, completing the proof of Theorem $C^*$.

\subsection{Stable branches and unstable branches}\label{subsecbranches}

In this subsection, we consider a homeomorphism $\fonc{\widetilde{f}}{\Ao}{\Ao}$ isotopic to the identity, commuting with the
translation $T$ and satisfying hypotheses $(H_1)$ and $(H_2)$. For $i\in\{0,1\}$, we define by $\widetilde{A}_i$ the topological closed band delimited by $\Gamma_i$ and $\Gamma_{i+1}$. Following an idea of Birkhoff, in this subsection we will
construct \textit{unstable branches} and \textit{stable branches} of the band $\widetilde{A}_i$.

\subsubsection{Stable sets and unstable sets}

Let us begin with the definition of an \textit{unstable set} (resp. \textit{stable set}) of the band $\widetilde{A}_i$ for $\widetilde{f}$.

\begin{defini}
  We say that $\Lambda^-_i$ is an \textit{unstable set} of the band $\widetilde{A}_i$ for
  $\widetilde{f}$ if it satisfies the following properties:
  \begin{enumerate}
    \item the set $\Lambda^-_i$ is a closed subset of $\widetilde{A}_i$;
    \item we have $\Lambda^-_i\cap \Gamma_{i+1}\neq \emptyset$; and
    \item the set $\Lambda^-_i$ is negatively invariant under $\widetilde{f}$, i.e. $\widetilde{f}^{-1}(\Lambda^-_i)\subset \Lambda^-_i$.
  \end{enumerate}
 Similarly we say that $\Lambda^+_i$ is a \textit{stable set} of the band $\widetilde{A_i}$ for $\widetilde{f}$ if it satisfies:
  \begin{enumerate}
    \item the set $\Lambda^+_i$ is a closed subset of $\widetilde{A}_i$;
    \item we have $\Lambda^+_i\cap \Gamma_i\neq \emptyset$;  and
    \item the set $\Lambda^+_i$ is positively invariant under $\widetilde{f}$, i.e. $\widetilde{f}(\Lambda^+_i)\subset \Lambda^+_i$.
  \end{enumerate}
\end{defini}

We prove the following proposition.

\begin{propo}\label{propense}
  Let $\fonc{\widetilde{f}}{\Ao}{\Ao}$ be a homeomorphism isotopic to the identity, commuting with the translation $T$ and
  satisfying hypotheses $(H_1)$ and $(H_2)$. For $i\in\{0,1\}$ and for every $n\in \Z$, we define
$$ \Lambda_{n,i}^-:=  \widetilde{f}^n(\adhe( U^-_{\Gamma_i})) \cap \adhe( U^+_{\Gamma_{i+1}}) \;\; \text{ and } \Lambda_{n,i}^+ :=\adhe( U^-_{\Gamma_i}) \cap   \widetilde{f}^{-n}(\adhe( U^+_{\Gamma_{i+1}} )).  $$
The sets
 $$ \Lambda^-_i :=\bigcap_{n\in \N} \Lambda_{n,i}^-  \;\; \text{ and } \Lambda^+_i :=\bigcap_{n\in \N} \Lambda_{n,i}^+,$$
are respectively an unstable set and a stable set of the
  band $\widetilde{A}_i$ for $\widetilde{f}$.
\end{propo}

\begin{figure}[h!]
  \centering
    \includegraphics{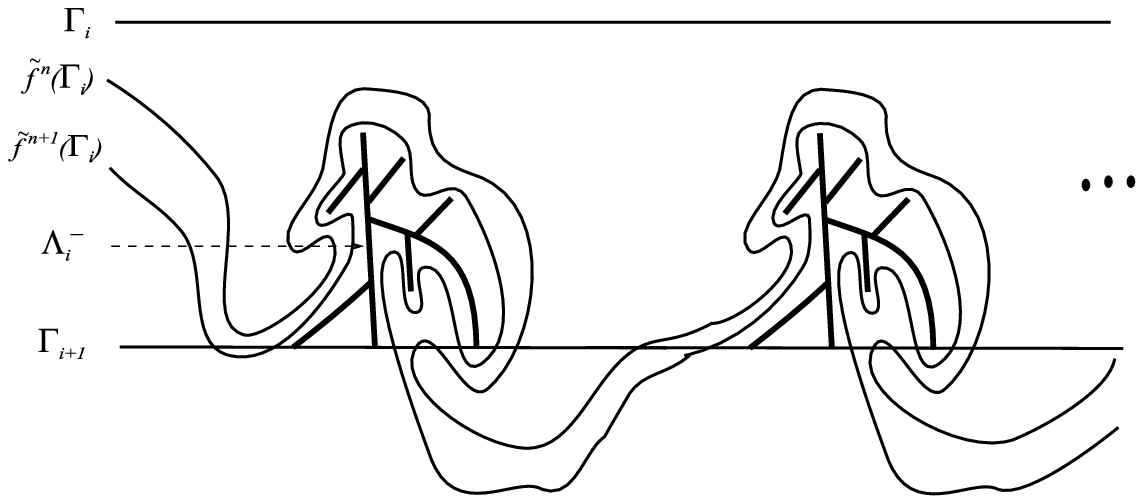}
  \caption{Unstable set}
  \label{fig:instableset}
\end{figure}

\begin{remar}\label{remaunstableset}
  The set $\Lambda^-_i$ consists of the points in $\widetilde{A}_i$ whose negative iterates under $\widetilde{f}$ remain in the band
  $\widetilde{A}_i$.
\end{remar}

\begin{proof}[Proof of Proposition \ref{propense}]
We will prove that $\Lambda^-_i$ is  an unstable set of the band $\widetilde{A}_i$ for
$\tilde{f}$ (one shows similarly that $\Lambda^+_i$ is a stable set of the band $\widetilde{A}_i$ for $\tilde{f}$). Item 1 of the above definition is a direct consequence of the definition of $\Lambda^-_i$ and the fact that $\Gamma_i$ is an attracting line. On the other hand, as the horizontal straight line $\Gamma_{i+1}$ is attracting, we have that $ \widetilde{f}^{-1}(\adhe(U^+_{\Gamma_{i+1}}))\subset \adhe(U^+_{\Gamma_{i+1}})$, and so
$$ \widetilde{f}^{-1}(\Lambda_i^-)= \bigcap_{n\in\N} \widetilde{f}^{n-1}( \adhe(U^-_{\Gamma_{i}})) \cap \widetilde{f}^{-1}(\adhe(U^+_{\Gamma_{i+1}}))\subset \Lambda_i^-.   $$
We have proved items $1$ and $3$ of the above definition. Next, let us prove item $2$, that is we prove that $\Lambda^-_i\cap\Gamma_{i+1}\neq \emptyset$ by contradiction. Suppose that for every $\tilde{z}\in \Gamma_{i+1}$, we have
$\widetilde{z}\notin \Lambda^-_i$. Then we can find the smallest integer $n_0=n_0(\widetilde{z})$ such that $\widetilde{f}^{-n_0(\widetilde{z})}
(\widetilde{z})\in U_{\Gamma_i}^+$. Since $\widetilde{f}$ is continuous and commutes with the translation $T$, we have that $n_0(\widetilde{z})\geq n_0(\widetilde{z}')$ for every $\widetilde{z}'$
close to $\widetilde{z}$ and that $n_0(\widetilde{z})=n_0(T(\widetilde{z}))$. By compactness, we can define the integer
 $$ n_0=\sup_{\widetilde{z}\in \Gamma_{i+1}} n_0(\widetilde{z})=\sup \{ n_0(\widetilde{z}): \widetilde{z}\in \Gamma_{i+1} \text{ and } p_1(\widetilde{z})\in [0,1]\}$$
 which satisfies $\widetilde{f}^{-n_0}(\Gamma_{i+1})\subset  U_{\Gamma_i}^+$. This contradicts hypothesis $(H_2)$, so item 2 is proved.
\end{proof}

\subsubsection{Stable branches and unstable branches}

Now, we state and prove the properties verified by unstable branches (we have the ``symmetric properties'' for stable branches) that we will use in the next subsections.

\begin{defini}
 For $i\in \{0,1\}$, we call \textit{unstable branch of the band $\widetilde{A}_i$} any connected component of the unstable set, $\Lambda^-_i$. For $x\in \Lambda^-_i$, we denote $\Lambda^-_i(x)$, and we say that it is \textit{the unstable branch of $x$}, the connected component of $\Lambda^-_i$ containing $x$.\\
 The \textit{stable branches} of $\widetilde{A}_i$ and the sets $\Lambda^+_i(x)$ are defined similarly.
\end{defini}

We have the following properties.

\begin{propo}\label{theobranches}
  Let $\fonc{\widetilde{f}}{\Ao}{\Ao}$ be a homeomorphism isotopic to the identity, commuting to the translation $T$ and satisfying hypotheses $(H_1)$ and $(H_2)$. If $x\in \Lambda^-_i$, then
  \begin{itemize}
    \item[(1)] the set $\Lambda^-_i(x)$ is a compact subset of $\widetilde{A}_i$;
    \item[(2)] we have $\Lambda^-_i(x)\cap \Gamma_{i+1}\neq \emptyset$;
    \item[(3)] we have $\widetilde{f}^{-1}(\Lambda^-_i(x))\subset \Lambda^-_i(\widetilde{f}^{-1}(x))$; and
    \item[(4)] there exists a real number $M^-_i>0$ (independent of $x$) such that $$\diam p_1(\Lambda^-_i(x)) <M^-_i.$$
  \end{itemize}
\end{propo}

The proof of Proposition \ref{theobranches} will use following lemma.

\begin{lemm}\label{lema127}
 For $x\in \Lambda_i^-$, we write $\Lambda_{n,i}^-(x)$ the connected component of $\Lambda_{n,i}^-$ containing $x$. Then
$$ \Lambda^-_i(x)=\bigcap_{n\in \N} \Lambda_{n,i}^-(x).$$
\end{lemm}
\begin{proof}[Proof of Lemma \ref{lema127}]
  First, for every $n$ in $\N$, from Proposition \ref{propense} we deduce that $\Lambda^-_i(x)\subset \Lambda^-_{n,i}(x)$. This implies that $\Lambda^-_i(x)\subset
  \bigcap_{n\in \N} \Lambda_{n,i}^-(x)$. To prove the reverse inclusion we need the following claim.

  \begin{claim}
    $(\Lambda^-_{n,i}(x))_{n\geq 1}$ is a decreasing sequence of compact and connected sets contained in $\Lambda_i^-$.
  \end{claim}
\begin{proof}[Proof of the claim] By definition for every integer $n\geq 1$, the set $\Lambda^-_{n,i}(x)$ is connected, and since $\Gamma_{i+1}$ is an attracting line, we deduce that $\Lambda^-_{n+1,i}(x)\subset \Lambda^-_{n,i}(x)$, i.e. the sequence is decreasing. Now, the compactness follows of the facts that $\widetilde{f}^{n}(\Gamma_0)\cap \Gamma_2\neq \emptyset$ (hypothesis $(H_2)$) and that $\widetilde{f}$ commutes with $T$. This proves the claim.
\end{proof}

 From the claim the intersection
 $\cap_{n\in\N}\Lambda^-_{n,i}(x)$ is a compact, connected set contained in $\Lambda_i^-$ which contains $x$. We conclude that $\cap_{n\in\N}\Lambda^-_{n,i}(x)\subset \Lambda^-_i(x)$. This completes the proof of the lemma.
\end{proof}

\begin{proof}[Proof of Proposition \ref{theobranches}]
\textit{Proof of item (1).} We proved (see the claim above) that $\Lambda^-_{n,i}(x)$ is a compact set for every integer $n\geq 1$, so
it provides the compactness of $\Lambda^-_i(x)$. Moreover by Proposition \ref{propense}, $\Lambda^-_{i}(x)$ is contained in $\widetilde{A}_i$.\\

\textit{Proof of item (2).} We know that
$\left(\Lambda^-_{n,i}(x)\right)_{n\geq 1}$ is a decreasing sequence of compact and connected sets. Moreover for every integer $n\geq 1$,
$\Lambda^-_{n,i}(x)\cap \Gamma_{i+1}\neq \emptyset$. Therefore $\left(\Lambda^-_{n,i}(x)\cap \Gamma_{i+1}\right)_{n\geq 1}$ is a decreasing sequence of compact non-empty sets and so by Proposition \ref{propense}, we have that
   $$\Lambda^-_i(x)\cap \Gamma_{i+1}= \bigcap_{n\in\N}\Lambda^-_{n,i}(x)\cap \Gamma_{i+1}\neq \emptyset.$$
This proves item (2).\\

\textit{Proof of item (3).} By Proposition \ref{propense}, we have that $\widetilde{f}^{-1}(\Lambda^-_i)\subset \Lambda^-_i$. By connectedness, we deduce that $\widetilde{f}^{-1}(\Lambda^-_i(x))\subset \Lambda^-_i(\widetilde{f}^{-1}(x))$.\\

\textit{Proof of item (4).} Let $x$ be in $\Lambda_i^-$. Since by Lemma \ref{lema127}, we have that $\Lambda_i^-(x)$ is contained in $\Lambda_{1,i}^-(x)$, it suffices check the property for $\Lambda_{1,i}^-$. Let $\widetilde{Q}_i:=\{\widetilde{z}\in \widetilde{A}_i: 0\leq p_1(\widetilde{z})\leq 1\}$ be and let
$$  M_0:=\sup_{\widetilde{z}\in \widetilde{A}_i} \abs{ p_1(\widetilde{f}(\widetilde{z}))-p_1(\widetilde{z})  }.  $$ First, we claim that $\diam p_1(\widetilde{f}(\widetilde{Q}_i))\leq 2M_0+1.$ Indeed, let $\widetilde{z}, \widetilde{z}'\in \widetilde{Q}_i$. Then
$$ \abs{ p_1(\widetilde{f}(\widetilde{z}))-p_1(\widetilde{f}(\widetilde{z'}))  }\leq 2M_0+1.$$

On the other hand, since $\widetilde{f}$ commutes with $T$ and $\Gamma_{i+1}$ is an attracting line, we have
$$  \Lambda_{1,i}^- \subset \widetilde{f}(\adhe(U_{\Gamma_i}^-))\cap \adhe(U_{\Gamma_{i+1}}^+) \subset \bigcup_{n\in \Z} (T^{n}(\widetilde{f}(\widetilde{Q}_i))\cap \adhe(U_{\Gamma_{i+1}}^+)).    $$
Hence, if $x\in \Lambda_{1,i}^-$, then there is an integer $n$ such that $x\in T^n(\widetilde{f}(\widetilde{Q}_i))\cap \adhe(U_{\Gamma_{i+1}}^+)$. Let $X$ be the connected component of $T^n(\widetilde{f}(\widetilde{Q}_i))\cap \adhe(U_{\Gamma_{i+1}}^+)$ containing $x$. We remark that by Keréj\'art\'o result (see\cite{ly97}) the boundary of $X$ is the union of two arcs $\widetilde{f}(\Gamma_i')$ where $\Gamma_i'\subset \Gamma_i $ and $\Gamma'_{i+1}\subset \Gamma_{i+1}$ . We have the following claim.

\begin{claim}
  $\Lambda_i^-(x)$ is contained in $X$.
\end{claim}
\begin{proof}
  By contradiction. If $\Lambda_i^-(x)$ is not contained in $X$, then by connectedness $\Lambda_i^-(x)$ meets $\widetilde{f}(\Gamma_i')$, because $\Lambda_i^-(x)$ is contained in $\widetilde{A}_i$ (item (1)). Using item (3), we deduce
  $$ \emptyset \neq \widetilde{f}^{-1}(\Lambda_i^-(x)) \cap \Gamma'_i \subset  \Lambda_i^-(\widetilde{f}^{-1}(x)) \cap \Gamma_i \subset \Lambda_i^- \cap \Gamma_i   .$$ This contradicts the definition of $\Lambda_i^-(x)$.
\end{proof}
 By the claim and the first point we deduce that

 $$ \diam(p_1(\Lambda_{1,i}^-(x)) \leq \diam(X) \leq \diam(T^n(\widetilde{f}(Q_i))) \leq 2M_0+1 .$$
\end{proof}

\subsection{Consequence of Hypothesis $(H_3)$}\label{subsecconsehypH3}
In this subsection, we consider a homeomorphism $\fonc{\widetilde{f}}{\Ao}{\Ao}$ isotopic to the identity, commuting with the translation $T$ and satisfying hypotheses ($H_1$)-($H_3$). Using hypothesis $(H_3)$ we will study the dynamics of unstable branches (we have the symmetric properties for stable branches).\\

Let us begin by recalling that for $i\in\{0,1\}$, the inclusion $\Theta(\widetilde{A}_i)\subset \widetilde{A}_i$ implies that $\rho_{\Theta(\widetilde{A}_i)}(\widetilde{f})\subset \rho_{\widetilde{A}_i}(\widetilde{f})$, but we cannot use hypothesis $(H_3)$ to conclude that $\rho_{\widetilde{A}_0}(\widetilde{f})\subset (0,+\infty)$ and $\rho_{\widetilde{A}_1}(\widetilde{f})\subset (-\infty,0)$. However, using the fact that the straight lines $\Gamma_i$ and $\Gamma_{i+1}$ are attracting, we can prove that the above inclusions are true.

\begin{propo}\label{propinclusion1}
 For $i\in \{0,1\}$, we have
  $$ \rho_{\widetilde{A}_i}(\widetilde{f})\subset \rho_{mes}(\Theta(\widetilde{A}_i),\widetilde{f})=\conv{\rho_{\Theta(\widetilde{A}_i)}(\widetilde{f})}. $$
  In particular $\rho_{\widetilde{A}_0}(\widetilde{f})\subset (0,+\infty)$ and $\rho_{\widetilde{A}_1}(\widetilde{f})\subset (-\infty,0)$.
\end{propo}
\begin{proof}
Recall that we are using the notation $\rho_{\widetilde{A}_i}(\widetilde{f})$ instead of $\rho_{\pi(\widetilde{A}_i)}(I)$.  Let $\rho$ be an element in $\rho_{\widetilde{A}_i}(\widetilde{f})$. By definition and the fact that $\gamma_i=\pi(\Gamma_i)$ and $\gamma_{i+1}=\pi(\Gamma_{i+1})$ are $f$-free, there exist a sequence of points
$\left(z_k\right)_{k\in\N}$ and a sequence of integers $\left(n_k\right)_{k\in\N}$ satisfying:
\begin{itemize}
  \item[(i)] for every $k\in\N$ and integer $j\in \{0,\cdots, n_k \}$ we have  $f^{j}(z_k)\in A_i=\pi(\widetilde{A}_i) $;
  \item[(ii)] the sequence $(n_k)_{k\in\N}$ converges to $+\infty$; and
  \item[(iii)] the sequence $(\rho_{n_k}(z_k))_{k\in\N}$ converges to $\rho$.
\end{itemize}
Let us consider the sequence of Borel probability measures $\suii{\mu}{k}$, where for every integer $k$, $\mu_k$ is defined as:

$$ \mu_{k}:=\frac{1}{n_k}\sum_{i=0}^{n_k-1} \delta_{f^{i}(z_k)},  $$ where $\delta_{z}$ ($z\in \A$) denotes the \textit{Dirac's measure at $z$}.
 By item (i), these measures have their supports included in $A_i$, and so changing the sequence $\suii{\mu}{k}$ by one of its subsequence, we may suppose that
 the sequence $\suii{\mu}{k}$ converges. Let $\mu$ be its limit. We know that the measure $\mu$ is invariant (under $f$), and so its support $\supp(\mu)$ is
 contained in the maximal $f$-invariant set of $A_i$, denoted by $\Theta(A_i)$. Moreover
$$ \rho_{n_k}(z_k)=\int \rho_1\;d\mu_k  \rightarrow \int \rho_1\;d\mu=\rho(\mu).$$
From item (iii) and the uniqueness of the limit, we deduce that $\rho$ belongs to $\rho_{mes}(\Theta(\widetilde{A}_i),\widetilde{f})$. This proves that  $\rho_{\widetilde{A}_i}(\widetilde{f})\subset \rho_{mes}(\Theta(\widetilde{A}_i),\widetilde{f})$. On the other hand, by Lemma \ref{lemaensK}, we know that $\rho_{mes}(\Theta(\widetilde{A}_i),\widetilde{f})=\conv{\rho_{\Theta(\widetilde{A}_i)}(\widetilde{f})}$. Thus, using hypothesis $(H_3)$ we have that $\rho_{\widetilde{A}_0}(\widetilde{f})\subset (0,+\infty)$ and $\rho_{\widetilde{A}_1}(\widetilde{f})\subset (-\infty,0)$, because $ \rho_{\widetilde{A}_i}(\widetilde{f})$ is a compact set. This completes the proof of the proposition.
\end{proof}

As an immediate consequence of the above proposition (and in particular using that $0$ does not belong to $\rho_{\widetilde{A}_0}(\widetilde{f})$) we have the following lemma.
\begin{lemm}\label{lemmaK}
  For every real number $M>0$, there exists $n_0=n_0(M)\in\N$ such that for every compact set $\widetilde{K}$ included in $\widetilde{A}_0$, with
  $\diam p_1(\widetilde{K})\leq M$, and for every integer $n$, $\abs{n}\geq n_0$ we have $$ \widetilde{f}^n(\widetilde{K})\cap \widetilde{K}=\emptyset.$$
\end{lemm}
\begin{proof}
  Suppose by contradiction that there are a real number $M_0>0$, and a sequence of integers $\suii{n}{k}$ which tends to $+\infty$ and points
  $\widetilde{z}_k\in \widetilde{A}_0$ such that
  $\widetilde{f}^{n_k}(\widetilde{z}_k)\in \widetilde{A}_0$, and $$\abs{p_1(\widetilde{f}^{n_k}(\widetilde{z}_k))-p_1(\widetilde{z}_k)}\leq  M_0.$$
 Hence,
  $$ \abs{\rho_{n_k}(\widetilde{z}_k, \widetilde{f})}:=\frac{1}{n_k}\abs{p_1(\widetilde{f}^{n_k}(\widetilde{z}_k))-p_1(\widetilde{z}_k)}\leq
  \frac{M_0}{n_k}.$$
  Letting $k$ tends to infinity, one deduces that $0$ belongs to $\rho_{\widetilde{A}_0}(\widetilde{f})$. This contradicts  Proposition \ref{propinclusion1}, completing the proof.
\end{proof}

In the sequel, for every real number $M$, and $i\in\{0,1\}$ we write
$$ L_{\widetilde{A}_i}(M):=\{\widetilde{z}\in \widetilde{A}_i: p_1(\widetilde{z})<M\} \quad  \text{ and } \quad R_{\widetilde{A}_i}(M):=\{\widetilde{z}\in \widetilde{A}_i: p_1(\widetilde{z})>M\}. $$

\begin{lemm}\label{lemmaxinfini}
 Let $x\in \Theta(\widetilde{A}_0)$. For every real number $M$, there exists $n_0\in\N$ such that for every integer $n\geq n_0$ we have $$ \widetilde{f}^{-n}(x)\subset L_{\widetilde{A}_0}(-M) \quad \text{ and } \quad  \widetilde{f}^{n}(x)\subset R_{\widetilde{A}_0}(M). $$
\end{lemm}
\begin{proof}
In both cases the proof is analogous, so we only prove the first inclusion.
Let  $$M_0:=\sup_{\widetilde{z}\in\widetilde{A}_0} \abs{p_1(\widetilde{f}(\widetilde{z}))-p_1(\widetilde{z})}.$$
Let $M$ be a positive real number such that $2M >M_0$ and $-M< p_1(x)<M$. Let us put $\widetilde{K}= \adhe(R_{\widetilde{A}_0}(-M))\cap \adhe(L_{\widetilde{A}_0}(M))$. By Lemma \ref{lemmaK}, there exists $n_0\in\N$ such that for every integer $n\geq n_0$ we have $\widetilde{f}^{-n}(x)\notin \widetilde{K}$, i.e. $\widetilde{f}^{-n}(x)\in L_{\widetilde{A}_0}(-M)\cup R_{\widetilde{A}_0}(M)$. Suppose that there are two integers $n''>n'\geq n_0$ such that
$\widetilde{f}^{-n'}(x)\in L_{\widetilde{A}_0}(-M)$ and $\widetilde{f}^{-n''}(x)\in R_{\widetilde{A}_0}(M)$. The following claim contradicts the choice of $n_0$ (see Figure \ref{fig:proofclaim}).

\begin{figure}[h!]
  \centering
    \includegraphics{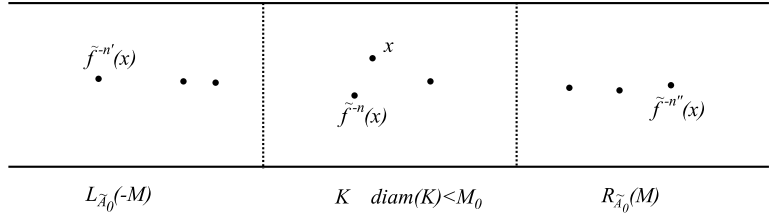}
  \caption{Orbit of $x$}
  \label{fig:proofclaim}
\end{figure}

\begin{claim}\label{equalemmaxinfini}
  There is an integer  $n''>n>n'\geq n_0$ such that $\widetilde{f}^{-n}(x) \in \widetilde{K}$.
\end{claim}
\begin{proof}
Let $$ n-1:=\max\{m\in \{n',\cdots,n'' \} :  \widetilde{f}^{-m}(x)\in L_{\widetilde{A}_0}(-M)  \}.$$
(From the choice of $n'$ and $n''$ such integer exists). By definition of $n$, we have $p_1(\widetilde{f}^{-n+1}(x))<-M\leq p_1(\widetilde{f}^{-n}(x)),$ and so in order to prove the claim it suffices to check that $ \widetilde{f}^{-n}(x)<M $.
Recall that for every integer $m$, the point $\widetilde{f}^{m}(x)$ is in $\widetilde{A}_0$ and so
$$p_1(\widetilde{f}^{-n}(x))= p_1(\widetilde{f}^{-n}(x)) -p_1(\widetilde{f}^{-n+1}(x)) + p_1(\widetilde{f}^{-n+1}(x)) \leq M_0 - M < M.$$
This proves the claim.
 \end{proof}
We conclude that the sequence $(f^{-n}(x))_{n\geq n_0}$ is contained either in $L_{\widetilde{A}_0}(-M)$ or in $R_{\widetilde{A}_0}(M)$. Since $\rho_{\Theta(\widetilde{A}_0)}(\widetilde{f})\subset (0,+\infty)$, the second inclusion cannot possible, completing the proof of the lemma.
\end{proof}

\begin{lemm}\label{lemma}
  Let $x\in \Theta(\widetilde{A}_0)$. Then the sequence $\{\widetilde{f}^{-n}(\Lambda_0^-(x))\}_{n\in\N}$ is uniformly bounded to the right, that is, for every real number $M$, there exists $n_0\in\N$ such that for every integer $n\geq n_0$ we have $$ \widetilde{f}^{-n}(\Lambda_0^-(x))\subset L_{\widetilde{A}_0}(M).$$
\end{lemm}
\begin{proof}
Let $n\in \N$. Since $\widetilde{f}^{-n}(\Lambda_0^-(x))\subset \Lambda_{0}^-(\widetilde{f}^{-n}(x))$ Proposition \ref{theobranches} provides us a real number $M_0^-$ (independent of $n$) such that for every $\widetilde{z}\in \Lambda_{0}^-(x)$ we have $$\abs{ p_1(\widetilde{f}^{-n}(\widetilde{z}))-p_1(\widetilde{f}^{-n}(x)) }<M_0^-.$$ On the other hand, by Lemma \ref{lemmaxinfini}, we have that $\lim_{n\to +\infty} p_1(\widetilde{f}^{-n}(x))=-\infty$. This completes the proof of the lemma.
\end{proof}

%

\subsection{Proof of Proposition \ref{prop61}}

In this section we prove Proposition \ref{prop61} reformulated as follows.

\begin{propo}[Proposition \ref{prop61}]\label{propimplytheoC}
Let $\fonc{\widetilde{f}}{\Ao}{\Ao}$ be a homeomorphism isotopic to the identity, commuting with the translation $T$, and satisfying hypotheses $(H_1)$ to $(H_3)$. For every $x\in \Theta(\widetilde{A}_0)$, there exist $y\in \Theta(\widetilde{A}_1)$ and a positive integer $n_0$ such that for every integer $n\geq n_0$, we have
  $$ \widetilde{f}^{n}(\Lambda_0^-(x)) \cap \Lambda_1^+(y)\neq \emptyset.$$
\end{propo}

 We note that if $\widetilde{z}\in  \widetilde{f}^{n}(\Lambda_0^-(x)) \cap \Lambda_1^+(y)$, then $z\in \Theta(\widetilde{A})$. By Lemma \ref{lemma} and its symmetric result the sequence of positive and negative iterates of $ \widetilde{z}$ go to the left. We prove to next subsection that this implies Theorem C$^*$.\\

\textit{Proof of Proposition \ref{propimplytheoC}}
Let $x\in \Theta(\widetilde{A}_0)$. Since $\Lambda_0^-(x)$ is a compact set, we can consider a real number $M$ such that $\Lambda_0^-(x)\subset L_{\widetilde{A}_0}(M)$. Let $M_1^+$ be the positive real constant provided by Proposition \ref{theobranches}, such that
$$  \max_{y'\in \Theta(\widetilde{A}_1)} \diam p_1(\Lambda_1^+(y')) < M_1^+.$$ Let us fix $y\in \Theta (\widetilde{A}_1)$ such that $p_1(y)>M+M_1^+$. We recall that by the choice of $y$, one has that
\begin{equation}\label{equaproofpropimplyTheoC1}
     \Lambda_0^-(x)\subset L_{\widetilde{A}_0}(M)   \quad \text{ and } \quad  \Lambda_1^+(y)\subset R_{\widetilde{A}_1}(M).
\end{equation}

\begin{lemm}\label{lemmproofpropimplyTheoC1}
  For every integer $n\geq 0$, we have $$  \widetilde{f}^{n}(\Gamma_1) \cap \Lambda_1^+(y)\neq \emptyset.$$
\end{lemm}
\begin{proof}
  Suppose by contradiction that there exists a positive integer $n'$ such that $  \widetilde{f}^{n'}(\Gamma_1) \cap \Lambda_1^+(y)=\emptyset$. This implies that $\Lambda_1^+(y)$ is contained in one of the connected components of $\Ao\setminus \widetilde{f}^{n'}(\Gamma_1)$. Since $\Gamma_1$ is an attracting line, i.e. $\widetilde{f}^{n'}(\Gamma_1)\subset U^-_{\Gamma_1}$, and since $\Lambda_1^+(y)$ meets $\Gamma_1$ (hypothesis $(H_1)$ and Proposition \ref{theobranches}), we conclude that $\Lambda_1^+(y)\subset U^+_{\widetilde{f}^{n'}(\Gamma_1)}$, and so $y\in U^+_{\widetilde{f}^{n'}(\Gamma_1)}=\widetilde{f}^{n'}(U_{\Gamma_1}^+)$. This contradicts the choice of $y$, because $y\in\Theta(\widetilde{A}_1)\subset \cap_{n\in\N} \widetilde{f}^n(U_{\Gamma_1}^-)$.
\end{proof}

From Proposition \ref{theobranches}, we know that $\Lambda_1^+(y)\cap \Gamma_1$ is a non-empty set. Let us fix $\widetilde{z}\in \Lambda_1^+(y)\cap \Gamma_1$ and consider the vertical line segment $\fonc{\alpha}{[0,1]}{\Ao}$ satisfying:
  \begin{itemize}
    \item $\alpha(0)=\widetilde{z}$;
    \item $\alpha(1)\in \Gamma_0$; and
    \item $\alpha((0,1))\subset \inte (\widetilde{A}_0)$.
  \end{itemize}
Let $R_{\widetilde{A}_0}(\alpha)$ (resp. $L_{\widetilde{A}_0}(\alpha)$) be connected component of $\widetilde{A}_0\setminus \alpha$ which is unbounded to the right (resp. to the left). Let $\widetilde{K}=\Lambda_0^-(x)\cup \alpha$. Since $\widetilde{K}$ is a compact subset of $\widetilde{A}_0$, by Lemma \ref{lemmaK} there is a positive integer $n_0$ such that for every integer $n\geq n_0$, we have
 \begin{equation}\label{equaproofpropimplyTheoC2}
 \widetilde{f}^n(\Lambda_0^-(x))\cap \alpha \subset \widetilde{f}^n(\widetilde{K})\cap  \widetilde{K}=\emptyset.
 \end{equation}
Moreover by Lemma \ref{lemmaxinfini}, we can assume that $n_0$ is large enough such that for every integer $n\geq n_0$.
 \begin{equation}\label{equaproofpropimplyTheoC3}
\widetilde{f}^n(x)\in R_{\widetilde{A}_0}(\alpha).
 \end{equation}
 From now on, we fix such integer $n_0$. Let us consider $\Lambda_0^\pm(x):=\Lambda_0^-(x)\cup \Lambda^+_0(x)$. This is a connected and compact subset of $\widetilde{A}_0$ which separates the band $\widetilde{A}_0$ (i.e. it meets both $\Gamma_0$ and $\Gamma_1$). We can define its right $R_{\widetilde{A}_0}(\Lambda_0^\pm(x))$ that is the connected component of $\widetilde{A}_0\setminus \Lambda_0^\pm(x)$ which is unbounded to the right.

Let us denote by
$$ \Gamma_1^r:=R_{\widetilde{A}_0}(\Lambda_0^\pm(x))\cap \Gamma_1 \quad  \text{ and } \quad \Gamma_1^l:= \Gamma_1\setminus \Gamma^r_1. $$

(See Figure \ref{fig:proofproposition}). Note that
 \begin{equation}\label{equaproofpropimplyTheoC4}
\Gamma_1^r \cap \Lambda_0^-(x) \subset R_{\widetilde{A}_0}(\Lambda_0^\pm(x)) \cap \Lambda_0^\pm(x)=\emptyset.
 \end{equation}

\begin{figure}[h!]
  \centering
    \includegraphics{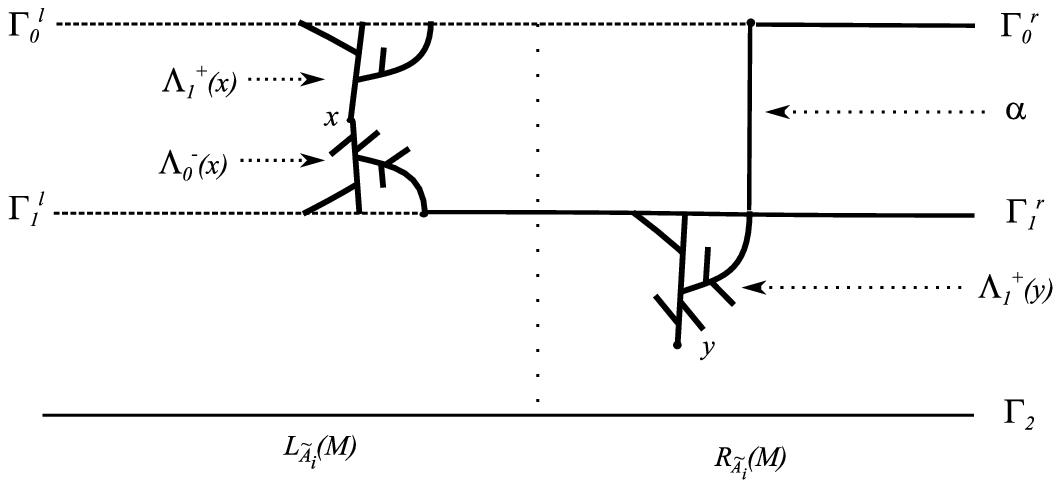}
  \caption{Proof Proposition \ref{prop61}}
  \label{fig:proofproposition}
\end{figure}

Proposition \ref{propimplytheoC} we will be a consequence of the following lemma.

\begin{lemm}\label{lemmproofpropimplyTheoC2}
  For every integer $n\geq n_0$, we have $$ (\widetilde{f}^n(\Lambda_0^-(x)) \cup \widetilde{f}^n(\Gamma^l_1)) \cap \Lambda_1^+(y)\neq \emptyset.$$
\end{lemm}
\begin{proof}
Fix an integer $n\geq n_0$. We have two cases. If $\widetilde{f}^n(\Gamma^r_1) \cap \Lambda_1^+(y)=\emptyset$, then by Lemma \ref{lemmproofpropimplyTheoC1} and the fact that $\Gamma_1=\Gamma_1^l \cup \Gamma_1^r$, we have that $\widetilde{f}^n(\Gamma^l_1) \cap \Lambda_1^+(y)\neq \emptyset$. This completes the proof in this case.\\

 Now suppose that $\widetilde{f}^n(\Gamma^r_1) \cap \Lambda_1^+(y)\neq \emptyset$. Let $\Gamma_0^l=\Gamma_0 \cap L_{\widetilde{A}_0}(\alpha)$ and $\Gamma_0^r=\Gamma_0\setminus \Gamma_0^l$. Let
 $$C_n:= \Lambda_{1}^+(y)\cup \widetilde{f}^{n}(\Gamma_1^r)\cup \alpha\cup \Gamma_0^r.$$
 Remark that $C_n$ is a connected set. Moreover we can assume that it does not meet $\widetilde{f}^{n}(\Lambda_0^-(x))$. Otherwise, $\widetilde{f}^{n}(\Lambda_0^-(x))\cap \Lambda^+_1(y)\neq \emptyset$, because $\widetilde{f}^{n}(\Lambda_0^-(x))$ meets none of $\widetilde{f}^{n}(\Gamma_1^r)$, $\alpha$ and $\Gamma_0$  (relations \eqref{equaproofpropimplyTheoC4} and \eqref{equaproofpropimplyTheoC2}, and remark \ref{remaunstableset} respectively), which finishes the proof in this case. \\
 Since $\Gamma_1$ is an attracting line, i.e. $\widetilde{f}^{n}(\Gamma_1) \subset U^-_{\Gamma_1}$ and $\Lambda_1^{+}(y)$ is contained in $\widetilde{A}_1$, we have that
 $$ \inte(R_{\widetilde{A}_0}(\alpha)) \subset \inte (\widetilde{A}_0)\setminus \alpha \subset \Ao\setminus C_n. $$
\begin{figure}[h!]
  \centering
    \includegraphics{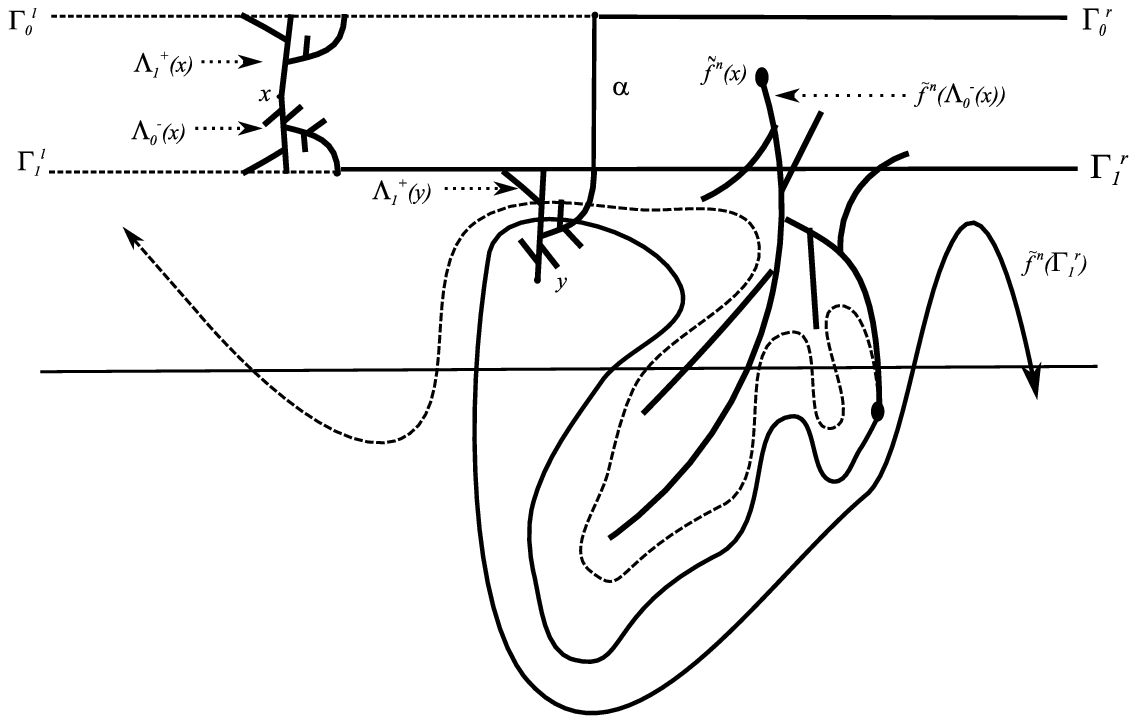}
  \caption{Proof Lemma \ref{lemmproofpropimplyTheoC2}}
  \label{fig:prooflemma}
\end{figure}

 Let $X_n$ be the connected component of $\Ao\setminus C_n$ that contains $\inte (R_{\widetilde{A}_0}(\alpha))$. Remark that since $C_n$ is a closed subset of $\Ao$, $X_n$ is an open set of $\Ao$ and $\fron(X_n)$ is contained in $C_n$. Hence $X_n$ is bounded to the left, because $\Lambda_1^+(y)\cup \alpha$ and $\widetilde{f}^n(\Gamma_1^r)$ are bounded to the left. On the other hand, since by relation \eqref{equaproofpropimplyTheoC3} $\widetilde{f}^n(x)\in R_{\widetilde{A}_0}(\alpha)$ and $\widetilde{f}^n(\Lambda_0^-(x))$ does not meet $C_n$, we have that $\widetilde{f}^n(\Lambda_0^-(x))$ is contained in $X_n$. Therefore the curve $\widetilde{f}^n(\Gamma_1^l)$ which is unbounded to the left must meet $\fron(X_n)$.
 As $\fron(X_n)\subset C_n$, it follows that $$\widetilde{f}^n(\Gamma_1^l)\cap C_n \neq \emptyset.$$ Since $\Gamma_1$ is an attracting line and since  $\Gamma_1^l\cap \Gamma_1^r=\emptyset$ and $\widetilde{f}$ is a homeomorphism, we have that $  \widetilde{f}^n(\Gamma_1)\cap (\alpha \cup \Gamma_0)=\emptyset$ and $  \widetilde{f}^n(\Gamma_1^l)\cap  \widetilde{f}^n(\Gamma_1^r)=\emptyset$ respectively. Therefore, we conclude that $\widetilde{f}^n(\Gamma_1^l)\cap \Lambda_1^+(y)\neq \emptyset$. This completes the proof of the lemma.

\end{proof}

\begin{proof}[End of the proof of Proposition \ref{propimplytheoC}]
Suppose by contradiction that there exists a sequence of integers $\suii{n}{k}$ which converges to $+\infty$ such that
$$  \widetilde{f}^{n_k}(\Lambda_0^-(x)) \cap \Lambda_1^+(y)=\emptyset.$$
From Lemma \ref{lemmproofpropimplyTheoC2} we obtain a sequence of points $\suii{\widetilde{z}}{k}$ with $\widetilde{z}_k\in \Gamma_1^l$ and $\widetilde{f}^{n_k}(\widetilde{z}_k)\in \Lambda_{1}^+(y)$. Recall that by $\eqref{equaproofpropimplyTheoC1}$, we have that $\Gamma_1^l\subset L_{\widetilde{A}_0}(M)$ and $\Lambda_1^+(y)\subset R_{\widetilde{A}_1}(M)$. Hence $ p_1(\widetilde{z}_k)<p_1(\widetilde{f}^{n_k}(\widetilde{z}_k))$. If $\rho$ is a limit point of the sequence $(\rho_{n_k}(\widetilde{z}_k))_{k\in\N}$, where
$$ \rho_{n_k}(\widetilde{z}_k):=\frac{1}{n_k}(p_1(\widetilde{f}^{n_k}(\widetilde{z}_k))-p_1(\widetilde{z}_k)),$$ then $\rho\geq 0$. This implies that $0\leq \rho\in \rho_{\widetilde{A}_1}(\widetilde{f})$, contradicting  Proposition \ref{propinclusion1}. This completes the proof of Proposition \ref{propimplytheoC}.
\end{proof}


\subsection{End of the proof of Theorem C*\label{theoremC4}}\label{subsecendprooftheoC}

 Let $n_0\in \N$ and let $\widetilde{z}\in \widetilde{f}^{n_0}(\Lambda_0^-(x)) \cap \Lambda_1^+(y)$ given by Proposition \ref{propimplytheoC}. From the definitions of $\Lambda_0^-(x)$ and $\Lambda_1^+(y)$ and the fact that $\Gamma_0$ and $\Gamma_2$ are $\widetilde{f}$-free, we deduce that $z\in \Theta(\widetilde{A})$. Moreover by Lemma \ref{lemma} and its symmetric property for stable branches, we know that
  \begin{equation}\label{equa1}
    \lim_{n\to+\infty } p_1(\widetilde{f}^{-n}(\widetilde{z}))=-\infty \quad \text{ and } \quad  \lim_{n\to+\infty } p_1(\widetilde{f}^{n}(\widetilde{z}))=-\infty.
  \end{equation}
Let us consider  $$M_0:=\sup_{\widetilde{z}'\in\widetilde{A}} \abs{p_1(\widetilde{f}(\widetilde{z}'))-p_1(\widetilde{z}')}.$$ Let $k$ be a positive integer and let us define $M:=p_1(\widetilde{z})-kM_0$. In a similar way as in the proof of the claim in Lemma \ref{lemmaxinfini}, we prove that there are two positive integers $n^+$ and $n^-$ such that
  $$ M-M_0\leq p_1(\widetilde{f}^{-n^-}(\widetilde{z}))<M  \text{  and  }  M-M_0\leq p_1(\widetilde{f}^{n^+}(\widetilde{z}))<M.$$
Therefore by the choice of $M$ and $M_0$, we deduce that
 \begin{itemize}
   \item $n^++n^-\geq 2k$, and
   \item $$ \abs{p_1(\widetilde{f}^{n^+}(\widetilde{z}))-p_1(\widetilde{f}^{-n^-}(\widetilde{z}))}\leq M_0.$$
 \end{itemize}
Since $k$ can be chosen arbitrarily large, this implies that $0$ belongs to $\rho_{\Theta(\widetilde{A})}(\widetilde{f})$. This is what we needed to prove.

\bibliographystyle{alpha}
\bibliography{bibliografie2}

  Institut de Mathématiques de Jussieu - Paris Rive Gauche,
  UPMC, 4 place Jussieu, Case 247, 75252 Paris Cedex 5.\\
  E-mail address: jonathan.conejeros@imj-prg.fr

\end{document}